%% file: ConvgOfiALMforNLP.tex
\documentclass[smallextended,final]{svjour3}
\usepackage[letterpaper,top=2in, bottom=1.5in, left=1in, right=1in]{geometry}

\usepackage{epsfig,epsf,fancybox}
\usepackage{amsmath}
\usepackage{mathrsfs}
\usepackage{amssymb}
\usepackage{amsfonts}
\usepackage{graphicx}
\usepackage{color}
\usepackage[linktocpage,pagebackref,colorlinks,linkcolor=blue,anchorcolor=blue,citecolor=blue,urlcolor=blue,hypertexnames=false]{hyperref}
\usepackage{boxedminipage}
\usepackage{stmaryrd}
\usepackage{multirow}
\usepackage{booktabs}
\usepackage{cite}
\usepackage{float}
\usepackage[ruled,vlined,linesnumbered]{algorithm2e}
\usepackage{algpseudocode}

\input{macros.tex}

\begin{document}

\title{Augmented Lagrangian based first-order methods for convex-constrained programs with weakly-convex objective}

\titlerunning{First-order ALM for constrained weakly convex problems}

\author{Zichong Li \and Yangyang Xu}

\institute{Z. Li, Y. Xu \at Department of Mathematical Sciences, Rensselaer Polytechnic Institute, Troy, NY 12180\\
\email{\{liz19, xuy21\}@rpi.edu}}

\date{\today}

\maketitle

\begin{abstract}
First-order methods (FOMs) have been widely used for solving large-scale problems. A majority of existing works focus on problems without constraint or with simple constraints. Several recent works have studied FOMs for problems with complicated functional constraints. In this paper,  
we design a novel augmented Lagrangian (AL) based FOM for solving problems with non-convex objective and convex constraint functions. The new method follows the framework of the proximal point (PP) method. On approximately solving PP subproblems, it mixes the usage of the inexact AL method (iALM) and the quadratic penalty method, while the latter is always fed with estimated multipliers by the iALM. We show a complexity result of $O(\vareps^{-\frac{5}{2}}|\log\vareps|)$ for the proposed method to achieve an $\vareps$-KKT point. This is the best known result. Theoretically, the hybrid method has lower iteration-complexity requirement than its counterpart that only uses iALM to solve PP subproblems, and numerically, it can perform significantly better than a pure-penalty-based method. Numerical experiments are conducted on 
nonconvex linearly constrained quadratic programs. The numerical results demonstrate the efficiency of the proposed methods over existing ones.

\vspace{0.3cm}


\vspace{0.3cm}


\end{abstract}


\section{Introduction}
First-order methods (FOMs) have been extensively used for solving large-scale problems, partly due to their low  per-iteration computational cost that usually scales linearly in problem dimension. Roughly speaking, FOMs only use the gradient and/or function value information of a problem and possibly other simple operations such as the projection onto a box set. A majority of existing works on FOMs study unconstrained problems, or those with easy-to-project and/or linear constraint set. Several recent works have focused on problems with nonlinear functional constraints, under both convex and nonconvex settings.

In this paper, we consider the nonlinear program:
\begin{equation}\label{eq:ccp}
f_0^*:=\Min_{\vx\in\RR^n} \big\{f_0(\vx), \st \vA\vx=\vb,\ \vf(\vx)\le \vzero\big\},
\end{equation}
where 
$\vA\in\RR^{l\times n}$ and $\vb$ are respectively given matrix and vector, and $\vf=(f_1,\ldots,f_m):\RR^n\to\RR^m$ is a vector function. We assume convexity and differentiability of $f_i$ for each $i=1,\ldots,m$, but $f_0$ could be nonconvex and nondifferentiable. The formula in \eqref{eq:ccp} is rather general. Any convex programs with finite constraints can be written into \eqref{eq:ccp}. Examples include linearly constrained (convex or nonconvex) quadratic programming (LCQP), convex quadratically constrained quadratic programming (QCQP), and the Neyman-Pearson classification problem \cite{scott2005neyman, rigollet2011neyman}.

\subsection{Augmented Lagrangian and penalty methods}
On solving a nonlinear functional constrained problem, the augmented Lagrangian method (ALM) \cite{hestenes1969multiplier, powell1969method} and penalty methods are perhaps the most classic and popular methods. The classic AL function of \eqref{eq:ccp} is:
\begin{equation}\label{eq:aug-fun}
\cL_\beta(\vx,\vy,\vz) = \textstyle f_0(\vx) + \vy^\top(\vA\vx-\vb) + \frac{\beta}{2}\|\vA\vx-\vb\|^2 + \frac{1}{2\beta}\left( \left\|[\vz+\beta\vf(\vx)]_+\right\|^2- \|\vz\|^2\right),
\end{equation}
where $\beta>0$ is the penalty parameter, $\vy$ and $\vz$ are the multiplier vectors, and $[\va]_+$ denotes a vector taking component-wise positive part of $\va$. If $f_i$ is convex for each $i=0,1,\ldots,m$, then $\cL_\beta$ is convex about $\vx$ and concave about $(\vy,\vz)$; see \cite[Lemma 1]{xu2019iter-ialm} for example. 
The ALM, at each iteration, updates the primal variable $\vx$ by minimizing $\cL_\beta$ with $(\vy,\vz)$ fixed and then performs a dual gradient ascent step to the multipliers. If the multipliers are always fixed to a zero vector, the ALM will reduce to a pure-penalty method.

The updates to $\vy$ and $\vz$ only require the matrix-vector multiplication and the function value of $\vf$. To update $\vx$, one can, in principle, apply any unconstrained optimization methods. We will focus on large-scale (possibly nonsmooth) problems, whose Hessian matrices do not exist or are too expensive to compute. 
Hence, it is beneficial to apply one FOM to approximately solve primal subproblems, and the resulting algorithm will be an AL-based FOM. 

\subsection{Literature review} 
In this subsection, we review related works on FOMs for solving \eqref{eq:ccp}. Some of existing FOMs are based on the ALM framework or the Lagrangian method, while some others are not. 

\vspace{0.3cm}

\noindent\textbf{Lagrangian-based FOMs.}~~On affinely constrained problems, i.e., in the form of \eqref{eq:ccp} with $m=0$, \cite{lan2016iteration-alm} analyzes the iteration complexity of an iALM whose primal-subproblems are approximately solved by an optimal FOM. For a smooth convex problem, \cite{lan2016iteration-alm} shows that $O(\vareps^{-\frac{7}{4}})$ gradient evaluations are sufficient to produce an $\vareps$-KKT point (see Definition~\ref{def:eps-kkt} below). The result is improved to $O(\vareps^{-1}|\log\vareps|)$ by applying the AL-based FOM to a perturbed strongly convex problem. While the result in \cite{lan2016iteration-alm} is established with constant penalty parameter $\beta=O(\vareps^{-1})$, the work \cite{liu2019nonergodic} uses $\beta=O(\vareps^{-\frac{1}{2}})$ and establishes the complexity result  $O(\vareps^{-2})$ to produce an $\vareps$-solution 
$(\bar{\vx},\bar{\vy})$ satisfying $\Vert \vA \bar{\vx}-\vb \Vert \leq \sqrt{\vareps}$ and $\max_\vx \left\{\langle \nabla g(\vx)+\vA^{\top} \bar{\vy}, \bar{\vx}-\vx \rangle + h(\bar{\vx})-h(\vx)\right\} \leq \vareps$, 
where $f_0=g+h$ has been assumed. The results in \cite{lan2016iteration-alm} have been extended to conic convex programming in \cite{aybat2013augmented, lu2018iteration}. The recent work \cite{xu2019iter-ialm} analyzed an AL-based FOM for solving \eqref{eq:ccp}. For convex problems, it shows a complexity result $O(\vareps^{-1})$ to produce an $\vareps$-optimal solution, and for strongly convex problems, the result is improved to $O(\vareps^{-\frac{1}{2}}|\log\vareps|)$. Different from the nonergodic result that we will show, the produced $\vareps$-optimal solution in \cite{xu2019iter-ialm} is the average of all iterates, i.e., the result is in an ergodic sense. Through accelerating a dual gradient ascent (DGA) method, \cite{nedelcu2014computational} can also obtain a complexity result of $O(\vareps^{-\frac{1}{2}}|\log\vareps|)$ for affinely constrained strongly convex problems. However, for convex problems, the complexity result of either the nonaccelerated or accelerated DGA is $O(\vareps^{-2})$ to produce an $\vareps$-optimal solution. Different from iALM, a proximal iALM is studied in \cite{li2019-piALM}, and it added a proximal term to each ALM subproblem. Although similar complexity results can be achieved, we notice that a proximal version of iALM can perform worse than the original iALM. A linearized ALM is analyzed in \cite{xu2017first}. It simply does one single gradient update to the primal variable before renewing the dual variables. For convex problems, an $O(\vareps^{-1})$ complexity result in an ergodic sense is established to obtain an $\vareps$-optimal solution.

The aforementioned works are all AL-based. There are also FOMs based on the ordinary Lagrangian function. For example, \cite{necoara2013rate} proposed an inexact first-order dual method and its accelerated version by using the ordinary Lagrangian dual function. To produce an $\vareps$-optimal solution, the accelerated method needs $O(\vareps^{-\frac{1}{2}})$ outer iterations and solves each primal subproblem to an accuracy $O(\vareps^{\frac{3}{2}})$. Hence, for smooth strongly convex problems, its total complexity can be $O(\vareps^{-\frac{1}{2}}|\log\vareps|)$, while for convex problems, its complexity is worse than $O(\vareps^{-1})$. Another example is \cite{hamedani2018primal}, which proposed a method for solving general convex-concave saddle point (SP) problems. It is remarked that a conic convex program can be formulated as an equivalent SP problem by the ordinary Lagrangian function, and the proposed method can be directly applied and satisfies all required conditions for convergence. In addition,  \cite{hamedani2018primal} shows that $O(\vareps^{-1})$ gradient and function evaluations are sufficient to produce a solution with $\vareps$-primal-duality gap.   

Besides solving convex problems, AL-based FOMs have also been analyzed for non-convex problems. A recent work \cite{sahin2019inexact}\footnote{An $\tilde{O}(\vareps^{-3})$ complexity is claimed in Corollary 4.2 in~\cite{sahin2019inexact}. However, this complexity is based on an existing result that was not correctly referred to. The authors claimed that the complexity of solving each nonconvex composite subproblem is $O(\frac{\lambda_{\beta_k}^2 \rho^2}{\vareps_{k+1}})$, which should be $O(\frac{\lambda_{\beta_k}^2 \rho^2}{\vareps_{k+1}^2})$; see~\cite{sahin2019inexact} for the definitions of $\lambda_{\beta_k}, \rho, \vareps_{k+1}$. Using the correctly referred result and following the same proof in \cite{sahin2019inexact}, we get a total complexity of $\tilde{O}(\vareps^{-4})$.} shows a complexity result of $O(\vareps^{-4})$ to produce an $\vareps$-KKT point for problems with nonconvex objective and nonconvex equality constraints. A key assumption that \cite{sahin2019inexact} uses is a regularity condition, which can ensure to control the primal residual by increasing the penalty parameter.

\vspace{0.3cm}

\noindent\textbf{Penalty-based FOMs.} Penalty methods are also commonly used for solving functional constrained problems. If the multiplier vectors are kept \emph{zero}, then the ALM can be viewed as a penalty method. For affinely constrained conic convex programs, \cite{lan2013iteration-pen} gives a first-order quadratic penalty method and establishes a complexity result of $O(\vareps^{-1}|\log\vareps|)$ to produce an $\vareps$-KKT point. For affinely constrained nonconvex problems, the recent work \cite{kong2019complexity-pen} proposes a first-order quadratic penalty method. By geometrically increasing the penalty parameter, it shows a complexity result of $O(\vareps^{-3})$ to obtain an $\vareps$-KKT point. This result has been improved in \cite{lin2019inexact-pp} to $O(\vareps^{-\frac{5}{2}}|\log\vareps|)$ for problems with weakly-convex objective and convex constraints. The order is the best known and matches with the result established in this paper. However, the method in \cite{lin2019inexact-pp} is pure-penalty-based and numerically performs worse than an AL-based method (see the experiments in section~\ref{sec:experiment}). For problems with weakly-convex objective and weakly-convex constraints, \cite{lin2019inexact-pp} gives an $O(\vareps^{-4})$ complexity result, which can be improved to $O(\vareps^{-3})$ if a strong Slater's condition or a regularity condition holds, as shown in \cite{ma2019proximally, boob2019proximal}.

\vspace{0.3cm}

\subsection{Contributions and new results}
\begin{itemize}
\item We propose a hybrid method of multipliers to solve nonconvex problems, for which we assume that the objective is weakly-convex but the constraint functions are still convex. The method is in the framework of the proximal point (PP) method. Utilizing the weak-convexity, we add to the objective a quadratic term with the current iterate as the prox-center, and this leads to a strongly-convex PP subproblem. On solving the PP subproblems, our method mixes the first-order iALM and a first-order penalty method with estimated multipliers,  neither of which is new but similar to existing iALM or penalty methods. \color{black} We break the whole algorithm into multiple stages. Since iALM is generally more efficient than a penalty method, we use it in an initial stage and also at the end of each following stage to estimate the multipliers. With the estimated multipliers, the penalty method can also perform well, and within each stage, the estimated multipliers are fixed so that we can control the change of primal iterates. 

\item Assuming the Slater's condition, we establish an $O(\vareps^{-\frac{5}{2}}|\log\vareps|)$ complexity result to produce an $\vareps$-KKT point. Our result improves nearly by an order of $\vareps^{-\frac{1}{2}}$ over that in \cite{kong2019complexity-pen}, which considers affinely constrained nonconvex problems. It matches with the result established in \cite{lin2019inexact-pp} for a pure-penalty based first-order method. Numerical experiments on nonconvex quadratic programs demonstrate the advantage of our method over the penalty based method. 

\item Over the course of our analysis, we establish an $O(\vareps^{-\frac{1}{2}}|\log\vareps|)$ complexity result of the first-order iALM to produce an $\vareps$-KKT point of a strongly convex problem, which has a composite structured objective and smooth constraints. The complexity result differs from an existing lower bound by $|\log\vareps|$ and is nearly optimal. In addition, the $\vareps$-KKT point can be obtained at an actual primal iterate. This improves the ergodic result in \cite{xu2019iter-ialm}.
\end{itemize}

\subsection{Notation, definitions, and assumptions}
We denote $[n]$ as the set $\{1,\ldots,n\}$. For any two vectors $\va$ and $\vb$ in $\RR^n$, $\va\ge \vb$ means $a_i\ge b_i$ for any $i\in[n]$, $\max\{\va,\vb\}$ denotes the vector by component-wise maximization. $[\va]_+$ is short for $\max\{\va,\vzero\}$. Given $M>0$, $\cB_M$ represents a Euclidean ball with center at origin and radius $M\ge 0$. For a closed convex set $X$, $\cN_X(\vx)$ denotes the normal cone of $X$ at $\vx$.
We denote the distance function between a vector $\vx$ and a set $\cX$ as $\dist(\vx,\cX) = \min_{\vy \in \cX} \Vert \vx-\vy \Vert$. We let $\cL_0$ be the ordinary Lagrangian function of \eqref{eq:ccp}, namely,
\begin{equation}\label{eq:OL}
\cL_0(\vx,\vy,\vz) = f_0(\vx) + \vy^\top (\vA \vx-\vb) + \vz^\top \vf(\vx). 
\end{equation}
We use $\vp=(\vy,\vz)$ for the dual variable and let 
\begin{equation}\label{eq:def-cP}\cP=\{(\vy,\vz)\in\RR^l\times \RR^m: \vz\ge\vzero\}
\end{equation}
be the dual feasible region. For any $\beta>0$, the augmented dual function is defined as
\begin{equation}\label{eq:aug-dual-fun}
d_\beta(\vp) = \min_{\vx}\cL_\beta(\vx,\vp), \text{ for }\vp\in\cP, 
\end{equation}
and the ordinary Lagrangian dual function is $d_0(\vp) = \min_{\vx} \cL_0(\vx,\vp),  \text{ for }\vp\in\cP$. By default, if $\vp\not\in \cP$, then $d_\beta(\vp) = d_0(\vp) = -\infty$.
The augmented dual problem is
$\max_{\vy, \vz}  d_\beta(\vy,\vz).$
It can be easily shown (c.f. \cite{rockafellar1973dual}) that for any $(\vx,\vy,\vz)$,
\begin{equation}\label{P1-eqn-2.5}
\textstyle \cL_\beta(\vx,\vy,\vz) = \min_{\vt\ge 0}\left\{ f_0(\vx) + \langle \vy, \vA\vx-\vb \rangle + \frac{\beta}{2} \left\Vert \vA\vx-\vb \right\Vert^2 +\frac{\beta}{2} \|\vf(\vx) + \vt\|^2+\vz^\top  \big(\vf(\vx)+\vt\big)\right\}.
\end{equation}

\begin{definition}[$\vareps$-KKT point]\label{def:eps-kkt}
	Given $\vareps \geq 0$, a point $\vx \in \dom(f_0)$ is called an $\vareps$-KKT point to \eqref{eq:ccp} if there is a dual point $(\vy,\vz) \in \cP$ such that
	\begin{align}\label{eq:kkt}
	\textstyle \sqrt{\left\Vert \vA\vx-\vb \right\Vert^2 + \left\Vert [\vf(\vx)]_+\right\Vert^2 } \leq \vareps,\quad \dist\left(\vzero, \textstyle \partial f_0(\vx)+\vA^\top \vy+\sum_{i=1}^{m}z_i\nabla f_i(\vx) \right) \leq \vareps, \quad \textstyle \sum_{i=1}^{m} | z_i f_i(\vx) | \leq \vareps. 
	\end{align}
\end{definition}

Throughout the paper, we make the following assumptions. 
\begin{assumption}[composite structure and smooth constraints]\label{assump:composite}
	The objective admits a composite structure, i.e., $f_0=g+h$, where $g$ is differentiable and has a Lipschitz-continuous gradient, and $h$ is a simple closed convex function with a bounded domain, i.e.,
	\begin{equation}\label{eq:def-D}
	D =: \max_{\vx,\vx'\in \dom(h)}\|\vx-\vx'\| <\infty.
	\end{equation} 
	Also, $f_i$ is a convex and Lipschitz differentiable function for every $i=1,\ldots,m$, i.e., there are constants $L_0, L_1, \dots, L_m$ such that
	\begin{subequations}\label{lip}
		\begin{align}
		&\left\Vert \nabla g(\widehat{\vx})-\nabla g(\bar{\vx})\right\Vert \leq L_0 \left\Vert \widehat{\vx}-\bar{\vx}\right\Vert, \forall \, \widehat{\vx},\bar{\vx} \in \dom(h),\\ 
		&\left\Vert \nabla f_i(\widehat{\vx})-\nabla f_i(\bar{\vx})\right\Vert \leq L_i \left\Vert \widehat{\vx}-\bar{\vx}\right\Vert, \forall \,  \widehat{\vx},\bar{\vx} \in \dom(h), \forall i \in [m]. \label{lip-i}
		\end{align}
	\end{subequations}
\end{assumption}

By \eqref{eq:def-D} and \eqref{lip}, there must exist constants $B_1,\dots,B_m$ such that
\begin{subequations}\label{P3-eqn-33}
	\begin{align}
	\max \big\{|f_i(\vx)|,\left\Vert \nabla f_i(\vx) \right\Vert\big\} \leq B_i, \forall\, \vx \in \dom(h), \forall\, i \in [m],\\
	|f_i({\vx})-f_i({\vx'})| \leq B_i \left\Vert {\vx}-{\vx'} \right\Vert, \forall\, {\vx},{\vx'} \in \dom(h), \forall\, i \in [m].
	\end{align}
\end{subequations}

{
We assume the boundedness on $\dom(h)$ in order to make sure the existence of the constants $\{L_i\}$ and $\{B_i\}$ when $g$ and $\{f_i\}$ are continuously differentiable. This assumption can be relaxed if the problem is convex, as shown in \cite{li2019-piALM}. However, we cannot relax it for nonconvex problems, as we will also need it to uniformly bound $f_0$ in the proof of Theorem~\ref{thm:bd-num-stage}. This weakness limits the applicability of our proposed method. Nevertheless, if a bound on the solution of an underlying problem can be estimated, one can impose a bounded constraint and then safely apply our method.
}

\begin{assumption}[Slater's condition]\label{assump:slater}
	There exists a point $\vx_{\mathrm{feas}}\in\mathrm{int}(\dom(h))$ such that $\vA\vx=\vb$ and $f_i(\vx)< 0$ for all $i\in [m]$.
\end{assumption}

\begin{assumption}[weak convexity]\label{assump:wkcvx}
	The smooth function $g$ is $\rho$-strongly convex with $\rho>0$, i.e., $g+\frac{\rho}{2}\|\cdot\|^2$ is convex. 
\end{assumption}

\begin{assumption}\label{assump:uniform-bound}
	The nonsmooth part $h$ of the objective has subdifferential satisfying $\partial h(\vx) \subseteq \cN_{\dom(h)}(\vx)+\cB_M$ for any $\vx\in\dom(h)$. In addition, $\vA$ is row full-rank. Furthermore, there is a constant $B_0$ such that $|f_0(\vx)| \le B_0$ for any $\vx\in\dom(h)$.
\end{assumption}

 The assumption $\partial h(\vx) \subseteq \cN_{\dom(h)}(\vx)+\cB_M$ is not too restrictive. We include the term $\cB_M$ to cover a broad class of functions. If $h$ is an indicator function of a compact convex set $\cX$, we have $\partial h(\vx)= \cN_{\cX}(\vx)$ and thus $M=0$; if part of $h$ is a sparsity regularizer $\lambda\|\vx\|_1$, we can choose $M=\lambda\sqrt n$. \color{black} Note that many constants are assumed in Assumptions~\ref{assump:composite}--\ref{assump:uniform-bound}, but their values are not required, except an upper bound on the weak-convexity constant $\rho$, in order to implement our new method proposed in this paper. In addition, the interior feasible point $\vx_{\mathrm{feas}}$ does not need to be known either.

\subsection{Outline}
The rest of the paper is organized as follows. In section~\ref{sec:apg}, we give an adaptive accelerated proximal gradient method for solving strongly convex composite problems. Two AL-based FOMs are presented and analyzed in section~\ref{sec:subroutine} for strongly convex constrained problems. Our new method for \eqref{eq:ccp} and its complexity analysis are given in section~\ref{sec:hybrid}. 
Numerical results are given in section~\ref{sec:experiment}, and finally we conclude the paper in section~\ref{sec:conclusion}. 

\section{Adaptive accelerated proximal gradient method}\label{sec:apg}
The core subproblems to solve in our proposed algorithm are strongly-convex composite problems in the form of 
\begin{equation}\label{eq:comp-prob}
\Min_{\vx\in\RR^n}~ F(\vx):=G(\vx)+H(\vx),
\end{equation}
where $G$ is $\mu$-strongly convex and $L_G$-smooth, and $H$ is a closed convex function. 
With $\mu>0$, one can apply the adaptive accelerated proximal gradient method (AdapAPG) (c.f., \cite{lin2014adaptive})\color{black} \ to approximately solve \eqref{eq:comp-prob}. The pseudocode is given in Algorithm~\ref{alg2}, which uses line search shown in Algorithm~\ref{alg:ALS}. Notice that the algorithm assumes a lower estimate on the strong convexity constant $\mu$ but does not need to know the gradient Lipschitz constant $L_G$. It searches the Lipschitz constant by back-tracking.  We choose Algorithm~\ref{alg2} as the subroutine 
because it not only has the optimal convergence rate but also converges fast in practice. In particular, Algorithm~\ref{alg2} can perform well even if the strong convexity constant is overestimated or if the problem is not strongly convex. As a result, when Algorithm~\ref{alg2} is used as the inner subroutine for non-convex problems (discussed later in Section \ref{sec:hybrid}), we do not need to accurately estimate the weak-convexity constant and can still 
make the whole algorithm converge fast in practice. \color{black}

\begin{algorithm}[h]\label{alg2}  
	\caption{$\mathrm{AdapAPG}(G,H,\vx^{-2},\mu, \vareps, L_{\min},\gamma_1, \gamma_2)$}
{\small	
	\DontPrintSemicolon
	\textbf{Input:} $L_{\min} \ge \mu > 0$, increase rate $\gamma_1 > 1$, decrease rate $1 \le \gamma_2 \le 2 \gamma_1$, and error tolerance $\vareps > 0$.\;
	\textbf{Initialization:} $\alpha_{-1} = 1, L_0 = L_{\min}$.\;
	\Repeat{$G(\vx^{-1}) \le G(\vx^{-2}) + \langle \nabla G(\vx^{-2}),\vx^{-1}-\vx^{-2} \rangle + \frac{L_0}{2} \Vert \vx^{-1} - \vx^{-2} \Vert^2$}{
		$L_0 = \gamma_1 L_0$.\\
		$\vx^{-1} = \argmin_\vx \langle \nabla G(\vx^{-2}),\vx \rangle + \frac{L_0}{2}\Vert \vx-\vx^{-2} \Vert^2+H(\vx)$.
	} 
	$\vx^{0} = \vx^{-1}$.\\
	\For{$k=0,1,\ldots$}{
		$(\vx^{k+1},M_k,\alpha_k) = \mathrm{AccelLineSearch}(\vx^k,\vx^{k-1},L_k,\mu,\alpha_{k-1},\gamma_1)$.\\
		$L_{k+1} = \max\{L_{\min},\frac{M_k}{\gamma_2}\}$.\\
		\If{$\dist(\vzero, \partial F(\vx^{k+1})) \le \vareps$}{
			return $\vx^{k+1}$ and stop.
		}
	}
}	
\end{algorithm}
\begin{algorithm}[h]\label{alg:ALS}  
	\caption{$(\vx^{k+1},M_k,\alpha_k) = \mathrm{AccelLineSearch}(\vx^k,\vx^{k-1},L_k,\mu,\alpha_{k-1},\gamma_1)$}
	{\small	
		\DontPrintSemicolon
		\textbf{Initialization:} increase rate $\gamma_1 > 1$, and $L = \frac{L_k}{\gamma_1}$.\;
		\Repeat{$G(\vx^{k+1}) \le G(\vy^k)+\langle \nabla G(\vy^k),\vx^{k+1}-\vy^k \rangle+\frac{L}{2}\Vert \vx^{k+1}-\vy^k \Vert^2$.}{
			$L \leftarrow L \gamma_1$.\\
			$\alpha_k \leftarrow \sqrt{\frac{\mu}{L}}$.\\
			$\vy^k \leftarrow \vx^k + \frac{\alpha_k(1-\alpha_{k-1})}{\alpha_{k-1}(1+\alpha_k)}(\vx^k-\vx^{k-1})$.\\
			$\vx^{k+1} \leftarrow \argmin_{\vx} \langle \nabla G(\vy^k),\vx \rangle+\frac{L}{2}\Vert \vx-\vy^k \Vert^2+H(\vx)$.
		}
	$M_k = L$.
	}	
\end{algorithm}
\color{black}


Since $F$ is strongly convex, denote $\vx^*$ as the unique minimizer of $F$, and let $F^* = F(\vx^*)$. The next result expands Theorem 1 in~\cite{lin2014adaptive}. It gives the total complexity of Algorithm~\ref{alg2} to produce an $\vareps$-stationary point of \eqref{eq:comp-prob}. 

\begin{theorem}\label{thm:total-iter-apg}
	Let $E = \left\lceil \log_{\gamma_1} \frac{L_G}{\mu} \right\rceil $. Given $\vareps>0$, within at most $3E+6T-3$ evaluations of the objective value of $G$ and the gradient $\nabla G$, Algorithm~\ref{alg2} will output a solution $\vx^T$ that satisfies $\dist\big(\vzero,\partial F(\vx^T)\big) \le \vareps$, 
	where 
	\begin{equation}\label{eq:total-iter-apg}
	T= \left\lceil \sqrt{\frac{L_G \gamma_1}{\mu}}\log \left( \frac{2(1+\gamma_1)L_G \sqrt{ \frac{\gamma_1 L_G}{\mu} \Vert \vx^{-2}-\vx^* \Vert^2 + \Vert \vx^0-\vx^* \Vert^2 }}{\vareps} \right) \right\rceil .
	\end{equation}
\end{theorem} 

\begin{proof}
By the $L_G$ smoothness of $F$, the preprocessing step (line 3-6) of Algorithm \ref{alg2} gives us $F(\vx^0)-F^* \le \frac{\gamma_1 L_G}{2}\Vert \vx^{-1}-\vx^* \Vert^2$, which combined with \eqref{eq:total-iter-apg} yields
\begin{equation}\label{eq:thm1_as}
 \left(1-\frac{\mu}{L \gamma_1}\right)^\frac{T}{2} \le \frac{\vareps}{2L(1+\gamma_1)}\sqrt{ \frac{\mu}{2 \left( F(\vx^0)-F^*+\frac{\mu}{2} \Vert \vx^0-\vx^* \Vert^2 \right) } }. 
\end{equation}
	By Theorem 1 in~\cite{lin2014adaptive}, we have $\forall k \ge 0$,
\begin{align}
F(\vx^k)-F^* &\le \left( 1-\sqrt{\frac{\mu}{L \gamma_1}} \right)^k \left(F(\vx^0)-F^*+\frac{\mu}{2} \Vert \vx^0-\vx^* \Vert^2 \right),\\
\Vert \vy^k-\vx^* \Vert^2 &\le \left( 1-\sqrt{\frac{\mu}{L \gamma_1}} \right)^k \frac{2}{\mu} \left( F(\vx^0)-F^*+\frac{\mu}{2} \Vert \vx^0-\vx^* \Vert^2 \right).
\end{align} 
Since $F$ is $\mu$ strongly convex, we have
\begin{align*}
F(\vx^k) &\ge F^* + \frac{\mu}{2}\Vert \vx^k-\vx^* \Vert^2 \\
\Vert \vx^k-\vx^* \Vert^2 &\le \frac{2}{\mu} (F(\vx^k)-F^*) \\
&\le \left( 1-\sqrt{\frac{\mu}{L \gamma_1}} \right)^k \frac{2}{\mu} \left( F(\vx^0)-F^*+\frac{\mu}{2} \Vert \vx^0-\vx^* \Vert^2 \right),
\end{align*} 
and thus 
\begin{align}\label{eq:x-y}
\Vert \vx^{k+1}-\vy^k \Vert &\le \Vert \vx^{k+1}-\vx^* \Vert + \Vert \vy^k - \vx^* \Vert \nonumber \\
&\le 2 \left( 1-\sqrt{\frac{\mu}{L \gamma_1}} \right)^{\frac{k+1}{2}} \sqrt{\frac{2}{\mu} \left( F(\vx^0)-F^*+\frac{\mu}{2} \Vert \vx^0-\vx^* \Vert^2 \right)},
\end{align} 
which combined with \eqref{eq:thm1_as} yields 
\begin{equation}\label{eq:xT-yT}
\Vert \vx^T - \vy^{T-1} \Vert \le \frac{\vareps}{L(1+\gamma_1)}.
\end{equation}
Now by the $\vx$ update and stopping condition in Algorithm \ref{alg:ALS}, we have $\forall k \ge 0$, $\vzero \in \nabla G(\vy^k) + L_k(\vx^{k+1}-\vy^k) + \partial H(\vx^{k+1})$, so $L_k(\vy^k-\vx^{k+1}) + \nabla G(\vx^{k+1}) - \nabla G(\vy^k) \in \partial F(\vx^{k+1})$. Hence,
\begin{align*}
\dist(\vzero, \partial F(\vx^{k+1})) &\le \Vert L_k(\vy^k-\vx^{k+1}) + \nabla G(\vx^{k+1}) - \nabla G(\vy^k) \Vert\\
&\le (L_k + L_G) \Vert \vx^{k+1}-\vy^k \Vert\\
&\le (1+\gamma_1)L \Vert \vx^{k+1}-\vy^k \Vert,
\end{align*}
which combined with \eqref{eq:xT-yT} yields
\begin{equation*}
\dist(\vzero, \partial F(\vx^T)) \le (1+\gamma_1)L \Vert \vx^T-\vy^{T-1} \Vert \le \vareps.
\end{equation*}

It remains to show at most $3E+6T-3$ evaluations of $G$ and $\nabla G$ are needed to compute $\vx^T$.
First note that since $L_{\min} \ge \mu$ and $E = \left\lceil \log_{\gamma_1} \frac{L_G}{\mu} \right\rceil$, at most $E$ iterations are needed to increase $L_{\min}$ above $L_G$.
Recall that in the $k$-th iteration of Algorithm \ref{alg2}, the local smoothness estimate starts from $L_k$, and it decreases by a factor of $\gamma_2 \le 2 \gamma_1$ after every call of Algorithm \ref{alg:ALS} and increases by a factor of $\gamma_1$ every time when the stopping condition of Algorithm \ref{alg:ALS} fails. Note that every preprocessing iteration (line 3 to 6) requires $1$ evaluation of $G(\vx^{-1})$ in addition to $2$ pre-computations of $G(\vx^{-2})$ and $\nabla G(\vx^{-2})$, while every iteration of Algorithm \ref{alg:ALS}  needs $3$ evaluations of $G$ and $\nabla G$. Therefore, one preprocessing iteration with $E$ iterations of Algorithm \ref{alg:ALS} at $k = 0$ and two iterations of Algorithm \ref{alg:ALS} (because $\frac{\gamma_2}{\gamma_1} \le 2$) for all $k \ge 1$ are sufficient to generate $\vx^T$. This gives the upper bound on the evaluations of $G$ and $\nabla G$ as $2 + 1 + 3(E+2(T-1)) = 3E+6T-3$.
\end{proof}

\color{black}

\begin{remark}\label{rm:dist-x-pF}
	Since $F$ is $\mu$-strongly convex, we have $\| \tilde\nabla F(\vx)\|\ge\mu\|\vx-\vx^*\|$ for any $\vx$, where $\tilde\nabla F(\vx)$ is the any subgradient of $F$ at $\vx$, and $\vx^*$ is the minimizer of $F$. Hence, if $\dist(\vzero, \partial F(\vx)) \le \vareps$, then
	\begin{equation}\label{eq:gap-Fxt}
	F(\vx)-F(\vx^*) \le \big\langle \tilde\nabla F(\vx), \vx- \vx^* \big\rangle \le \frac{1}{\mu}\|\tilde\nabla F(\vx)\|^2 \le \frac{\vareps^2}{\mu},
	\end{equation}
namely, a near stationary point is also a near optimal solution when $F$ is strongly convex.	
\end{remark}


\section{Two methods of multipliers for strongly convex problems}\label{sec:subroutine}
Our proposed algorithm for \eqref{eq:ccp} will be designed under the framework of the proximal point method. It approximately solves a sequence of subproblems in the form of
\begin{equation}\label{barsub}
\Min_{\vx\in\RR^n} \big\{f_0(\vx)+\rho \Vert \vx-\bar\vx \Vert^2, \st \vA\vx=\vb, f_i(\vx)\le 0, i=1,\ldots, m\big\},
\end{equation}
where $\bar\vx$ is the primal-center of the subproblem. In this section, we give two methods as subroutines for solving \eqref{barsub}. One is an inexact ALM (named iALM in Algorithm~\ref{alg:ialm-stcvx}) and the other an inexact penalty method with estimated multipliers (named PenMM in Algorithm~\ref{alg:penalty}). Both iALM and PenMM are based on the augmented Lagrangian function. Their difference lies in whether the multipliers are updated or not during the iterations. 

We use two different subroutines for good numerical performance and also low complexity results. As we will discuss in section~\ref{sec:comp-ialm} that a complexity result of $O(\vareps^{-3}|\log\vareps|)$ is obtained if only iALM is used. This result is worse than $O(\vareps^{-\frac{5}{2}}|\log\vareps|)$ obtained by our hybrid method. In addition, we will see in section~\ref{sec:experiment} that the hybrid method performs significantly better than a pure-penalty method.

Throughout this section, we denote 
\begin{equation}\label{eq:bar-func}
\bar g(\vx) = g(\vx) + \rho\|\vx-\bar\vx\|^2;\ \bar\cL_0(\vx,\vp) = \cL_0(\vx,\vp) + \rho\|\vx-\bar\vx\|^2; \ \bar\cL_\beta(\vx,\vp) = \cL_\beta(\vx,\vp) + \rho\|\vx-\bar\vx\|^2, 
\end{equation}
where $\cL_0$ and $\cL_\beta$ are given in \eqref{eq:OL} and \eqref{eq:aug-fun} respectively, and $\rho$ is the weak convexity constant in Assumption~\ref{assump:wkcvx}.
All the three functions in \eqref{eq:bar-func} are $\rho$-strongly convex about $\vx$. In addition, under Assumption~\ref{assump:slater}, we have (c.f.: \cite[Thm 28.2]{rockafellar2015convex}) that any optimal solution $\bar\vx^*$ of \eqref{barsub} must be a KKT point, namely, there is $\bar\vp^*=(\bar\vy^*, \bar\vz^*)\in\RR^l\times\RR^m$ such that
\begin{align}\label{eq:kkt-system}
\vzero \in  \nabla \bar g(\bar\vx^*) + \partial h(\bar\vx^*) + \vA^\top \bar\vy^* + \sum_{i=1}^{m} z_i^*\nabla f_i(\vx^*),\ \bar\vz^* \geq \vzero; \ \vA\bar\vx^* = \vb, \vf(\bar\vx^*) \leq \vzero; \ (\bar\vz^*)^\top \vf(\bar\vx^*) = 0.
\end{align}


\subsection{An inexact augmented Lagrangian method for \eqref{barsub}}
In the framework of the ALM, the AL function needs to be minimized about the primal variable at every iteration. We apply the AdapAPG given in Algorithm~\ref{alg2} to approximately solve each ALM subproblem of \eqref{barsub}. For our purpose, we obtain every iterate as a near-stationary point of the AL function. The pseudocode is given in Algorithm~\ref{alg:ialm-stcvx}.

\begin{algorithm}[h] 
	{\small
		\caption{$(\beta_{\mathrm{out}},\vx_{\mathrm{out}},\vp_{\mathrm{out}})=\text{iALM}(\vareps,\beta_0,\sigma,\rho,\bar\vx, L_{\min}, \gamma_1, \gamma_2)$ for solving \eqref{barsub}}\label{alg:ialm-stcvx}
		\DontPrintSemicolon
		\textbf{Input:} tolerance $\vareps>0$, weak-convexity constant $\rho>0$,  $\beta_0>0$, $\sigma>1$, $\gamma_1 > 1$, $\gamma_2 \ge 1$, and $L_{\min} >0$;\;
		\textbf{Initialization:} choose $\vx^0\in\dom(h), \vy^0 = \vzero, \vz^0 = \vzero$;\; 
		\For{$k=0,1,\ldots,$}{
			Choose $\vareps_k = \bar{\vareps} = \sqrt\frac{\sigma-1}{\sigma+1}\frac{\vareps}2\min\{1,\sqrt{\rho}\}$;\;
			Apply Algorithm~\ref{alg2} to find 
			$ \vx^{k+1}=\mathrm{AdapAPG}\left(F-h, h, \vx^k, \rho,  \vareps_k, L_{\min}, \gamma_1, \gamma_2\right)$ with $F(\vx)=\bar\cL_{\beta_k}(\vx, \vp^k)$;
			\;  
			Update the dual variable $\vp=(\vy,\vz)$ with $\vy$ and $\vz$ respectively by 
			\begin{align}
			\vy^{k+1} =&~\vy^k  + \beta_k (\vA\vx^{k+1}-\vb),\label{eq:alm-y}\\
			\vz^{k+1}=&~\max\big\{\vzero,\ \vz^k + \beta_k \vf(\vx^{k+1})\big\}.\label{eq:alm-z}
			\end{align}
			Compute
			\begin{equation}\label{eq:check-stop}
			\mathrm{ERROR} = \max\left\{\frac{\|\vp^k\| + \|\vp^{k+1}\|}{\beta_k},\, \sum_{i=1}^m |z_i^{k+1} f_i(\vx^{k+1})| \right\};
			\end{equation}\;
			\textbf{if} $\mathrm{ERROR} \le \vareps$, \textbf{then} return $(\beta_{\mathrm{out}},\vx_{\mathrm{out}},\vp_{\mathrm{out}})=(\beta_k,\vx^{k+1},\vp^{k+1})$ and stop.\;
			Let $\beta_{k+1}=\sigma\beta_k$.
		}
	}
\end{algorithm}

 The stopping condition in Algorithm~\ref{alg:ialm-stcvx} is guaranteed to hold within $K$ iterations, where $K$ is given in Theorem \ref{thm:iter-eps-kkt}. \color{black} Once Algorithm~\ref{alg:ialm-stcvx} stops, the output $(\vx^{k+1},\vp^{k+1})$ must be an $\vareps$-KKT point of \eqref{barsub}.

\begin{theorem}\label{thm:stop-eps-sol}
	Suppose $k=K-1$ for some integer $K\ge 1$ when Algorithm~\ref{alg:ialm-stcvx} stops. Then the output $\vx^K$ is 
	an $\vareps$-KKT point of \eqref{barsub}. 
\end{theorem}

\begin{proof}
By the update of $\vy$ in \eqref{eq:alm-y}, it holds that  
\begin{equation}\label{eq:rel-y}
\left\Vert \vA\vx^{k+1}-\vb \right\Vert^2=\frac{1}{\beta_k^2}\|\vy^{k+1}-\vy^k\|^2.
\end{equation} Also, from the update of $\vz$ in \eqref{eq:alm-z}, we have
$f_i(\vx^{k+1}) \leq \frac{1}{\beta_k}(z_i^{k+1}-z_i^k)$, and thus
$[f_i(\vx^{k+1})]_+\le \frac{1}{\beta_k}|z_i^{k+1}-z_i^k|$ for any $i\in[m]$. Hence, 
\begin{equation}\label{eq:rel-z}
\left\Vert [\vf(\vx^{k+1})]_+ \right\Vert^2\le\frac{1}{\beta_k^2}\|\vz^{k+1}-\vz^k\|^2,
\end{equation}
which together with \eqref{eq:rel-y} gives 
\begin{equation}
		\textstyle	\left\Vert \big[\vf(\vx^{k+1})\big]_+ \right\Vert^2+ \left\Vert \vA\vx^{k+1}-\vb \right\Vert^2 \leq \frac{1}{\beta_k^2}\Vert \vp^{k+1}-\vp^k \Vert^2. \label{P2-eqn-4.13-sum}
		\end{equation}

Now by the update of $\vx^{k+1}$, we have $\dist\big(\vzero, \partial_\vx\bar\cL_{\beta_k}(\vx^{k+1},\vp^k)\big) \le \vareps_k$. Furthermore, by the update of $\vp$, it holds $\partial_\vx\bar\cL_{0}(\vx^{k+1},\vp^{k+1})=\partial_\vx\bar\cL_{\beta_k}(\vx^{k+1},\vp^k)$. Hence, $\dist\big(\vzero, \partial_\vx\bar\cL_{0}(\vx^{k+1},\vp^{k+1})\big)\le \vareps_k < \vareps$ for each $k$. In addition, we have $ \max\left\{\frac{\|\vp^{K-1}-\vp^{K}\|}{\beta_{K-1}},\, \sum_{i=1}^m |z_i^{K} f_i(\vx^{K})| \right\}\le \vareps$, as Algorithm~\ref{alg:ialm-stcvx} stops. Therefore, it holds that $\big\Vert [\vf(\vx^{K})]_+ \big\Vert^2+ \left\Vert \vA\vx^{K}-\vb \right\Vert^2\le \vareps^2$ by \eqref{P2-eqn-4.13-sum}, and $\vx^K$ is an $\vareps$-KKT point of \eqref{eq:ccp} by Definition~\ref{def:eps-kkt}.
\end{proof}

Below, we establish the iteration complexity result of Algorithm~\ref{alg:ialm-stcvx}. First, for a fixed $\vp$, let $F(\cdot) = \bar\cL_\beta(\,\cdot\,,\vp)$ and $G = F-h$. 
Then
\begin{equation*} 
\textstyle \nabla G(\vx)=\nabla \bar g(\vx)+\vA^\top  \vy+\beta \vA^\top  (\vA\vx-\vb)+\sum_{i=1}^{m}[z_i+\beta f_i(\vx)]_+ \nabla f_i(\vx),
\end{equation*} 
and by \cite[Lemma 5]{xu2019iter-ialm}, $\nabla G(\vx)$ is Lipschitz continuous on $\dom(h)$ with constant
\begin{equation}\label{P3-eqn-34}
\textstyle L(\vz,\beta) = L_0+2\rho+\beta \left\Vert \vA^\top  \vA \right\Vert+\sum_{i=1}^{m} (\beta B_i(B_i+L_i)+L_i|z_i|),
\end{equation}
where $\{L_i\}_{i=0}^m$ and $\{B_i\}_{i=1}^m$ are given in \eqref{lip} and \eqref{P3-eqn-33} respectively. 

Second, from the proof of \cite[Lemma 7]{xu2019iter-ialm}, it holds that if $\bar\cL_{\beta_t}(\vx^{t+1},\vp^t)\le \min_\vx\bar\cL_{\beta_t}(\vx,\vp^t)+e_t, \, \forall\, t\le k-1$ for an error sequence $\{e_t\}$, then
\begin{equation}\label{eq:bound-dual-xu}
\textstyle \|\vp^k-\bar\vp^*\|^2 \le \|\vp^0-\bar\vp^*\|^2 + 2\sum_{t=0}^{k-1}\beta_t e_t,\, \forall\, k\ge 0,
\end{equation}
where $\bar\vp^*$ is a dual solution satisfying the KKT conditions in \eqref{eq:kkt-system}. When $\vp^0=\vzero$, \eqref{eq:bound-dual-xu} implies
\begin{align}
&\textstyle \|\vp^k\|^2 \le 4\|\bar\vp^*\|^2+ 4\sum_{t=0}^{k-1}\beta_t e_t,\, \forall\, k\ge 0, \label{p-sq-bound}\\
&\textstyle \left\Vert \vp^k \right\Vert \leq 2 \left\Vert \bar\vp^* \right\Vert + \sqrt{2\sum_{t=0}^{k-1}\beta_t e_t},\, \forall\, k\ge 0. \label{p-bound}
\end{align}

\begin{theorem}[Iteration complexity of iALM]\label{thm:iter-eps-kkt}
	Under Assumptions~\ref{assump:composite} to \ref{assump:wkcvx}, let $\vareps\in(0,\frac{1}{2}]$. Then within $K$ iterations, Algorithm~\ref{alg:ialm-stcvx} will stop, and the output $\vx^K$ is an $\vareps$-KKT point of \eqref{eq:ccp}, where
	\begin{equation}\label{eq:K-max-iter-kkt}
	\textstyle K= \lceil\log_\sigma 
	\widehat C_\vareps\rceil + 1,\ \mathrm{with}\ \widehat C_\vareps=\max\left\{\frac{10\|\bar\vp^*\|^2}{\beta_0\vareps},\, \frac{8\|\bar\vp^*\|}{\beta_0\vareps},\, \frac{4}{\beta_0}\right\},
	\end{equation}
with $\bar\vp^*$ a dual solution satisfying the KKT conditions in \eqref{eq:kkt-system}.	
	In addition, the total number of gradient and function value evaluations	
$$T_{\mathrm{total}}\le 6\sqrt{\gamma_1}\left( \sqrt{\kappa}K+\widetilde{C}_{\vareps} \right)\log \left( \frac{2(1+\gamma_1)\widehat{L}_{\vareps}D\sqrt{\frac{\gamma_1 \hat{L}_\vareps}{\rho}+1}}{\vareps} \right) + \left( 3\widehat{E}+3 \right)K  ,$$
\color{black}
where
\begin{equation}\label{eq:def-Lc-kappac}
\begin{aligned}
&\kappa=\frac{L_0+2\rho+2\|\bar\vp^*\|\sqrt{\textstyle\sum_{i=1}^mL_i^2}}{\rho},\, L_c=\|\vA^\top\vA\|+\sum_{i=1}^m B_i(B_i+L_i),\\
& \widehat L_\vareps=L_0+2\rho+\beta_0 \sigma L_c \widehat C_\vareps +\sqrt{\sum_{i=1}^mL_i^2}\left(2\left\Vert \bar\vp^* \right\Vert +\frac{\vareps}{\sqrt 2}\sqrt{\beta_0 \widehat C_\vareps}\right),\\
& \widetilde{C}_{\vareps} = \frac{\sqrt{\beta_0 L_c \sigma^2 \widehat C_\vareps}}{\sqrt{\rho}(\sqrt{\sigma}-1)}+ \left(\sqrt{\sum_{i=1}^{m}L_i^2} \frac{\vareps \sqrt{\beta_0}}{\sqrt{2\sigma}} \right)^{\frac{1}{2}} \frac{\sigma^{\frac{1}{2}}\widehat C_\vareps^{\frac{1}{4}}}{\sqrt{\rho}(\sigma^{\frac{1}{4}}-1)} ,\\
& \widehat{E} = \left\lceil \log_{\gamma_1} \frac{\widehat{L}_\vareps}{\rho} \right\rceil.
\end{aligned}
\end{equation}
\end{theorem}

\begin{proof}
Applying the result in \eqref{eq:gap-Fxt}, we have $\bar\cL_{\beta_k}(\vx^{k+1},\vp^k)\le \min_\vx\bar\cL_{\beta_k}(\vx,\vp^k)+\frac{\vareps_k^2}{\rho}$. Hence, we obtain from \eqref{p-sq-bound} and \eqref{p-bound} with $e_t=\frac{\vareps_t^2}{\rho}$ that
\begin{equation}\label{p-bound-kkt}
\|\vp^k\|^2 \le 4\|\bar\vp^*\|^2+ 4\sum_{t=0}^{k-1}\beta_t \frac{\vareps_t^2}{\rho},\quad
\left\Vert \vp^k \right\Vert \leq 2 \left\Vert \bar\vp^* \right\Vert + \sqrt{2\sum_{t=0}^{k-1}\beta_t \frac{\vareps_t^2}{\rho}},\, \forall\, k\ge 0.
\end{equation}
Noticing $\sum_{t=0}^{k-1}\beta_t\vareps_t^2=\beta_0\bar\vareps^2\frac{\sigma^k-1}{\sigma-1}$, we have
\begin{equation}\label{eq:bd-for-feas2}
\frac{\|\vp^k\|+\|\vp^{k+1}\|}{\beta_k} \le \frac{4\|\bar\vp^*\|}{\beta_0\sigma^k} + \frac{\sqrt{2\bar\vareps^2/\rho}(\sqrt\sigma+1)}{\sqrt{\beta_0(\sigma-1)\sigma^k}}\le \frac{4\|\bar\vp^*\|}{\beta_0\sigma^k} + \frac{\vareps}{\sqrt{\beta_0\sigma^k}},
\end{equation}
where the first inequality holds because of the choice of $\bar\vareps$ and the fact $\sqrt\sigma + 1\le \sqrt{2\sigma+2}$.
Hence, with the choice of $K$ in \eqref{eq:K-max-iter-kkt}, we have $\frac{\|\vp^{K-1}\|+\|\vp^{K}\|}{\beta_{K-1}}\le \vareps$.

Since $\vz^k\ge\vzero$, it holds $\sum_{i=1}^{m} \left| z_i^k f_i(\vx^k) \right| = \sum_{i:z_i^k > 0} \left| z_i^k f_i(\vx^k) \right|$. In addition, notice from \eqref{eq:alm-z} that $z_i^k>0$ implies $z_i^k = z_i^{k-1}+\beta_{k-1}f_i(\vx^k)$. Hence, 
\begin{align}\label{eq:bd-compl-0}
\textstyle\sum_{i=1}^{m} \left| z_i^k f_i(\vx^k) \right| = \sum_{i:z_i^k > 0} \left| z_i^k \frac{z_i^k-z_i^{k-1}}{\beta_{k-1}} \right|
\leq &~ \textstyle\frac{1}{\beta_{k-1}} \left(\|\vz^k\|^2 + \frac{1}{4}\|\vz^{k-1}\|^2\right), 
\end{align}
which together with \eqref{p-bound-kkt} gives
\begin{equation}\label{eq:bd-compl-pre}
\sum_{i=1}^{m} \left| z_i^k f_i(\vx^k) \right|\le \frac{1}{\beta_{k-1}} \left(5\|\bar\vp^*\|^2+ 4\sum_{t=0}^{k-1}\beta_t \frac{\vareps_t^2} \rho + \sum_{t=0}^{k-2}\beta_t \frac{\vareps_t^2}\rho\right)= \frac{1}{\beta_0\sigma^{k-1}} \left(5\|\bar\vp^*\|^2+\frac{\beta_0\bar\vareps^2}{\rho}\frac{4\sigma^k+\sigma^{k-1}-5}{\sigma-1}\right).
\end{equation}
Plugging $\bar\vareps=\sqrt\frac{\sigma-1}{\sigma+1}\frac{\vareps}2\min\{1,\sqrt{\rho}\}$ into the above inequality and using $\vareps\le \frac{1}{2}$ yields
\begin{equation}\label{eq:bd-compl}
\sum_{i=1}^{m} \left| z_i^k f_i(\vx^k) \right|\le \frac{1}{\beta_0\sigma^{k-1}} \left(5\|\bar\vp^*\|^2+\frac{\beta_0\vareps(4\sigma^k+\sigma^{k-1}-5)}{8(\sigma+1)}\right)\le \frac{5\|\bar\vp^*\|^2}{\beta_0\sigma^{k-1}} + \frac{\vareps}{2}.
\end{equation}
Hence, by setting $K$ as that in \eqref{eq:K-max-iter-kkt}, we have $\sum_{i=1}^{m} \left| z_i^K f_i(\vx^K)\right|\le \vareps$. Therefore, Algorithm~\ref{alg:ialm-stcvx} must stop within $K$ iterations, and the output $\vx^K$ is an $\vareps$-KKT point of \eqref{eq:ccp} by Theorem~\ref{thm:stop-eps-sol}. Below, we estimate the total number of APG iterations.

From \eqref{eq:def-D} and Theorem~\ref{thm:total-iter-apg} with $\vareps=\vareps_k=\bar\vareps$, it suffices to produce $\vx^{k+1}$ by $3E_k + 6T_k - 3$ evaluations of objective and gradient, where $E_k = \left\lceil \log_{\gamma_1}\frac{L(\vz^k,\beta_k)}{\rho} \right\rceil$ and
\begin{equation}\label{eq:bd-iter-per-out2}
T_k= \left\lceil \sqrt{\frac{L(\vz^k,\beta_k) \gamma_1}{\rho}}\log \frac{2(1+\gamma_1)L(\vz^k,\beta_k) D \sqrt{ \frac{\gamma_1 L(\vz^k,\beta_k)}{\rho}+1 }}{\vareps} \right\rceil .
\end{equation}
Now, we obtain from \eqref{P3-eqn-34} and  \eqref{p-bound-kkt} that
\begin{align}\label{eq:bd-lzk-kkt}
L(\vz^k,\beta_k)\le &~ \textstyle L_0+2\rho+\beta_k L_c+\sqrt{\sum_{i=1}^mL_i^2}\left(2\left\Vert \bar\vp^* \right\Vert + \sqrt{2\sum_{t=0}^{k-1}\beta_t \vareps_t^2/\rho}\right)\cr
\le &~ \textstyle L_0+2\rho+\beta_0 L_c \sigma^k+\sqrt{\sum_{i=1}^mL_i^2}\left(2\left\Vert \bar\vp^* \right\Vert + \sqrt{\frac{\sigma^k-1}{2(\sigma+1)}\beta_0 \vareps^2}\right),
\end{align}
where in the second inequality, we have used $\vareps_t=\sqrt\frac{\sigma-1}{\sigma+1}\frac{\vareps}2\min\{1,\sqrt{\rho}\}$ for each $t$. Hence, $L(\vz^k,\beta_k)\le \widehat L_\vareps$ for each $k\le K-1$ by the fact $\sigma^k < \sigma \widehat C_\vareps,\,\forall\, k\le K-1$, and also 
\begin{align*}
\sum_{k=0}^{K-1}\sqrt{L(\vz^k,\beta_k)}\le &~ \sum_{k=0}^{K-1}\left( \left(L_0+2\rho+2\|\bar\vp^*\|\sqrt{\textstyle\sum_{i=1}^mL_i^2}\right)^\frac{1}{2}+\sqrt{\beta_0 L_c \sigma^k}+\left(\vareps\sqrt{\textstyle\sum_{i=1}^mL_i^2}\sqrt{\textstyle\frac{\beta_0\sigma^k}{2(\sigma+1)}}\right)^\frac{1}{2}\right)\cr
\le&~ \left(L_0+2\rho+2\|\bar\vp^*\|\sqrt{\textstyle\sum_{i=1}^mL_i^2}\right)^\frac{1}{2}K+\frac{\sqrt{\beta_0L_c\sigma^K}}{\sqrt{\sigma}-1}+\left(\vareps\sqrt{\textstyle\sum_{i=1}^mL_i^2}\sqrt{\textstyle\frac{\beta_0}{2(\sigma+1)}}\right)^\frac{1}{2}\frac{\sigma^{K/4}}{\sigma^{1/4}-1}.
\end{align*}
Summing up \eqref{eq:bd-iter-per-out2} over $k=0$ through $K-1$, plugging the above inequality,  and using $L(\vz^k,\beta_k)\le \widehat L_\vareps, \forall\,k\le K-1$ and $\sigma^K < \sigma^2\widehat C_\vareps$, we obtain the desired upper bound on $T_{\mathrm{total}}$ and complete the proof.
\end{proof}

\begin{remark}
	We set $\vareps_k=\sqrt\frac{\sigma-1}{\sigma+1}\frac{\vareps}2\min\{1,\sqrt{\rho}\}$ for the convenience of analysis. From \eqref{eq:bd-for-feas2} and \eqref{eq:bd-compl-pre}, we see that a complexity result of the same order can be obtained if $\vareps_k < \min\big\{\vareps,\, \sqrt{\frac{\vareps\rho(\sigma-1)}{4\sigma}}\big\}$ for each $k$. Although we analyze the iALM for solving \eqref{barsub}, the result $O(\vareps^{-\frac{1}{2}}|\log\vareps|)$ holds for any strongly convex problems with convex smooth constraints.
\end{remark}

\subsection{A penalty method with estimated multipliers for \eqref{barsub}}
In this subsection, we design a quadratic penalty method to approximately solve \eqref{barsub}. 
Numerically, an iALM can be significantly more efficient than a classic quadratic penalty method. Motivated by this, we propose a \textbf{pen}alty \textbf{m}ethod with estimated \textbf{m}ultipliers (PenMM), where the estimated multipliers are provided by iALM. The pseudocode is given in Algorithm~\ref{alg:penalty}. Notice that this algorithm is different from Algorithm~\ref{alg:ialm-stcvx} at the update to $\vy$ and $\vz$. PenMM always uses the initially provided multipliers in the update of $\vp$, while iALM uses the current estimated multipliers in the update. {In addition, the dual update in Line 6 is not needed, as the $\vareps$-dual feasibility is implied by the $\vareps_k$-stationarity of the solution $\vx^{k+1}$ by AdpAPG in Line 5. However, we keep $(\vy^k,\vz^k)$ for use in the analysis}. On solving \eqref{barsub}, the iALM and PenMM achieve similar complexity results. However, the use of PenMM in our proposed method for solving \eqref{eq:ccp} is critical to obtaining the $O(\vareps^{-\frac{5}{2}}|\log\vareps|)$ complexity result in order to produce an $\vareps$-KKT point.

\begin{algorithm}[h]
	{\small
		\caption{$(\beta_{\mathrm{out}}, \vx_{\mathrm{out}},\vp_{\mathrm{out}})=\mathrm{PenMM}(\vareps,\beta_0,\sigma,\rho,\bar\vx,\bar\vp, L_{\min}, \gamma_1, \gamma_2)$ for solving \eqref{barsub}}\label{alg:penalty}
		\DontPrintSemicolon
		\textbf{Input:} error tolerance $\vareps>0$, $\beta_0>0$, $\sigma>1$, weak-convexity constant $\rho>0$, $(\bar\vx, \bar\vp)\in \dom(h)\times \cP$, and $\gamma_1>1$, $\gamma_2\ge 1$,  $L_{\min}>0$;\;
		\textbf{Initialization:} let $\vx^0=\bar\vx, \vy^0=\bar\vy, \vz^0=\bar\vz$, $k = 0$;\;
		\While{$(\vx^k,\vy^k,\vz^k)$ is not an ${\vareps}$-KKT point of \eqref{barsub}}{
			Choose $\vareps_k = \bar{\vareps} = {\vareps}\min\{1,\sqrt{\rho}\}$;\;
			Apply Algorithm~\ref{alg2} to find $\vx^{k+1}=\mathrm{AdapAPG}(F-h, h, \rho, \vareps_k, L_{\min}, \gamma_1, \gamma_2)$ with $F(\cdot) = \bar\cL_{\beta_k}(\cdot,\bar\vp)$;\;
			Let $\vy^{k+1}=\bar\vy + \beta_k(\vA\vx^{k+1}-\vb)$ and $\vz^{k+1}=\left[\bar\vz+\beta_k \vf(\vx^{k+1})\right]_+$;\;
			Set $\beta_{k+1}=\sigma\beta_k$ and increase $k\gets k+1$. 
		} 
		\textbf{Output}: let $\vx_{\mathrm{out}}=\vx^k, \vp_{\mathrm{out}}=\vp^k$, and let $\beta_{\mathrm{out}}=\beta_{k-1}$ if $k\ge1$, and $\beta_{\mathrm{out}}=\beta_{0}$ otherwise.
	}
\end{algorithm}

Below, we show that if $\beta_k$ is sufficiently large, $(\vx^{k+1},\vy^{k+1},\vz^{k+1})$ must be an ${\vareps}$-KKT point of \eqref{barsub}, and thus the stopping condition in Algorithm~\ref{alg:penalty} must be satisfied after finitely many updates.

\begin{theorem}\label{thm:bd-betak}
	Suppose that the conditions in Assumptions~\ref{assump:composite} through \ref{assump:wkcvx} hold.  Let $\vareps\le \frac{3}{32}$ and $\bar\vp^*$ be a dual solution of \eqref{barsub} satisfying the conditions in \eqref{eq:kkt-system}. If for some $k$, it holds
	\begin{equation}\label{eq:cond-betak-pen}
	\textstyle  \beta_k\ge \bar\beta :=\max\left\{\frac{4\big(\|\bar\vp^*\|^2+\|\bar\vp^*-\bar\vp\|^2\big)}{\vareps},\, \frac{4\|\bar\vp^*-\bar\vp\|}{\vareps},\, 8\right\},
	\end{equation}
	then $(\vx^{k+1},\vy^{k+1},\vz^{k+1})$ must be an ${\vareps}$-KKT point of \eqref{barsub}.
\end{theorem}

\begin{proof}
First notice that $\partial_\vx \cL_{\beta_k}(\vx^{k+1},\bar\vy,\bar\vz) =\partial_\vx \cL_{0}(\vx^{k+1},\vy^{k+1},\vz^{k+1})$. Hence, it holds
\begin{equation}\label{eq:dual-feas-barsub}
\dist\left(\vzero, \partial f_0(\vx^{k+1})+2\rho(\vx^{k+1}-\bar\vx)+\vA^\top\vy^{k+1}+\sum_{i=1}^m z_i^{k+1}\nabla f_i(\vx^{k+1})\right) \le \vareps_k \le {\vareps}.
\end{equation}
Secondly, by the same arguments showing \eqref{P2-eqn-4.13-sum}, we have
\begin{equation}\label{eq:primal-feas-barsub}
\left\Vert \big[\vf(\vx^{k+1})\big]_+ \right\Vert^2+ \left\Vert \vA\vx^{k+1}-\vb \right\Vert^2 \leq \frac{1}{\beta_k^2}\Vert \vp^{k+1}-\bar\vp \Vert^2.
\end{equation}
Since $F$ is $\rho$-strongly convex, 
it follows from \eqref{eq:gap-Fxt} and $\vareps_k={\vareps}\min\{1,\sqrt\rho\}$ that 
\begin{equation}\label{eq:obj-dec-penalty}
\cL_{\beta_k}(\vx^{k+1},\bar\vp)+\rho\|\vx^{k+1}-\bar\vx\|^2\le \min_\vx \{\cL_{\beta_k}(\vx,\bar\vp)+\rho\|\vx-\bar\vx\|^2\} + {\vareps^2}, \end{equation}
and thus by \eqref{eq:bound-dual-xu} and noting that one single update is performed from $\bar\vp$ to $\vp^{k+1}$, we have
\begin{equation}\label{eq:bd-p-barsub}
\Vert \vp^{k+1}-\bar\vp^* \Vert^2\le \|\bar\vp^*-\bar\vp\|^2 + 2{\beta_k\vareps^2}.
\end{equation}
Therefore, it holds from \eqref{eq:primal-feas-barsub} that
\begin{equation}\label{eq:primal-feas-barsub-2}
\sqrt{\textstyle\left\Vert [\vf(\vx^{k+1})]_+ \right\Vert^2+ \left\Vert \vA\vx^{k+1}-\vb \right\Vert^2} \le \frac{1}{\beta_k}\left(\Vert \vp^{k+1}-\bar\vp^* \Vert+\|\bar\vp^*-\bar\vp\|\right)\le \frac{1}{\beta_k}\left(\textstyle 2\|\bar\vp^*-\bar\vp\|+\vareps\sqrt{ 2{\beta_k}}\right)\overset{\eqref{eq:cond-betak-pen}}\le {\vareps}.
\end{equation}
Thirdly, by the same arguments in \eqref{eq:bd-compl-0}, we have
$\sum_{i=1}^m\left|z_i^{k+1}f_i(\vx^{k+1})\right| \le \frac{1}{\beta_k} \left(\|\vz^{k+1}\|^2+\frac{1}{4}\|\bar\vz\|^2\right).
$ From \eqref{eq:bd-p-barsub}, it follows that $\|\vz^{k+1}\|^2\le 2\|\bar\vp^*\|^2+2\|\bar\vp^*-\bar\vp\|^2 + 4\beta_k \vareps^2.$ In addition, notice $\|\bar\vz\|^2\le \|\bar\vp\|^2\le 2\|\bar\vp^*\|^2+2\|\bar\vp^*-\bar\vp\|^2.$ 
Therefore, by the condition of $\beta_k$ in \eqref{eq:cond-betak-pen} and also $\vareps \le \frac{3}{32}$, we have $\sum_{i=1}^m\left|z_i^{k+1}f_i(\vx^{k+1})\right| \le {\vareps}$, which together with \eqref{eq:dual-feas-barsub} and \eqref{eq:primal-feas-barsub-2} implies the desired result.
\end{proof}

The next theorem gives an upper bound on the total number of gradient and function value evaluations that Algorithm~\ref{alg:penalty} needs to return the output.
\begin{theorem}[complexity results of PenMM]\label{thm:complexity-pen}
	Suppose that the conditions in Assumptions~\ref{assump:composite} through \ref{assump:wkcvx} hold. Let $\vareps\le \frac{3}{32}$. To return the output, the total number of APG iterations that Algorithm~\ref{alg:penalty} takes is	
	$$T_{\mathrm{total}}\le 6\alpha_K\sqrt{\frac{\gamma_1}{\rho}}\log \left( \frac{2(1+\gamma_1)\bar{L}D\sqrt{\frac{\gamma_1 \bar{L}}{\rho}+1}}{\vareps} \right) + \left(3\bar{E}+3\right)K
	,$$
	\color{black}
	where $\bar\beta$ is given in \eqref{eq:cond-betak-pen}, $L_c$ is defined in \eqref{eq:def-Lc-kappac}, and
	$$\textstyle K=\left\lceil \log_\sigma \frac{\bar\beta}{\beta_0}\right\rceil + 1, \quad \bar\vareps = {\vareps}\min\{1,\sqrt{\rho}\}, \quad \bar L= L_0+2\rho+L_c\sigma\bar\beta+\|\bar\vz\|\sqrt{\sum_{i=1}^mL_i^2},$$
	$$\alpha_K = K\sqrt{2\rho} + K\left(L_0+\Vert \bar{\vz} \Vert\sqrt{\sum_{i=1}^{m}L_i^2}\right)^{\frac{1}{2}}+\frac{\sigma \sqrt{\bar{\beta}L_c}}{\sqrt{\sigma}-1}, \bar{E} = \left\lceil \log_{\gamma_1}\frac{\bar{L}}{\rho} \right\rceil.$$
\end{theorem}

\begin{proof}
From the setting of $K$, it follows $\sigma^{K-1}\ge \frac{\bar\beta}{\beta_0}$ and $\sigma^{K-1} < \frac{\sigma\bar\beta}{\beta_0}$. Since $\beta_k=\beta_0\sigma^k$ for each $k\ge 0$, we have from Theorem~\ref{thm:bd-betak} that the stopping condition of Algorithm~\ref{alg:penalty} must be satisfied when $k\ge K-1$. 
From the definition of $L(\vz,\beta)$ in \eqref{P3-eqn-34}, we have for each $k\le K-1$ that
\begin{align}
L(\bar\vz,\beta_k)\le L_0+2\rho+\beta_k L_c+\|\bar\vz\|\sqrt{\sum_{i=1}^mL_i^2}=&~L_0+2\rho+\beta_0 L_c\sigma^k+\|\bar\vz\|\sqrt{\sum_{i=1}^mL_i^2}\label{eq:bd-Lz-1}\\
\le&~L_0+2\rho+L_c\sigma\bar\beta+\|\bar\vz\|\sqrt{\sum_{i=1}^mL_i^2}\label{eq:bd-Lz-2}.
\end{align}
Hence, by Theorem~\ref{thm:total-iter-apg} and \eqref{eq:def-D}, the total number of evaluations of objective and gradient is 111
\begin{equation}\label{eq:T-total-pen}
T_{\mathrm{total}}\le 6\left( \sum_{k=0}^{K-1} \sqrt{L(\bar{\vz},\beta_k)}\sqrt{\frac{\gamma_1}{\rho}}\log\frac{2(1+\gamma_1)\bar{L}D\sqrt{\frac{\gamma_1 \bar{L}}{\rho}+1}}{\vareps} \right)+\left(3\bar{E}+3\right)K.
\end{equation}

By \eqref{eq:bd-Lz-1}, it holds
\begin{align*}
\sum_{k=0}^{K-1}\sqrt{L(\bar\vz,\beta_k)}\le &~ K\left(\textstyle L_0+2\rho+\|\bar\vz\|\sqrt{\sum_{i=1}^mL_i^2}\right)^{\frac{1}{2}}+\sqrt{\beta_0L_c}\frac{\sigma^{\frac{K}{2}}-1}{\sqrt\sigma-1}\cr
\le &~ K\left(\textstyle L_0+2\rho+\|\bar\vz\|\sqrt{\sum_{i=1}^mL_i^2}\right)^{\frac{1}{2}}+\frac{\sigma \sqrt{\bar\beta L_c}}{\sqrt\sigma-1}.
\end{align*}
Using the fact $\sqrt{a+2\rho}\le \sqrt{a}+\sqrt{2\rho}$ for any $a\ge0$, and plugging the above inequality and \eqref{eq:bd-Lz-2} into \eqref{eq:T-total-pen}, we have the desired result and complete the proof.
\end{proof}

\section{A hybrid method of multipliers}\label{sec:hybrid}
In this section, we give a new algorithm for solving nonconvex problems in the form of \eqref{eq:ccp}. Our method is a hybrid of the iALM and PenMM. The use of iALM is to estimate the multipliers and to have fast convergence, while the PenMM is used for better complexity result in the worst case.

\subsection{Complexity result with pure iALM}\label{sec:comp-ialm} 
Suppose we start from $\vx^0\in\dom(h)$. Given any $\vareps>0$, for each $k\ge0$, we can apply Algorithm~\ref{alg:ialm-stcvx} to find an $\tilde\vareps$-KKT point $\vx^{k+1}$ of \eqref{barsub} with $\bar\vx=\vx^k$. Therefore, for each $k\ge0$, there is $\vp^{k+1}$ such that $\dist\big(\vzero, \partial\cL_0(\vx^{k+1},\vp^{k+1})+2\rho(\vx^{k+1}-\vx^k)\big) \le \tilde\vareps$, where $\cL_0$ is defined in \eqref{eq:OL}. Since the near-feasibility and near-complementarity conditions for \eqref{barsub} are the same as those for \eqref{eq:ccp}, $\|\vx^{k+1}-\vx^k\|\le \frac{\vareps}{4\rho}$ and $\tilde\vareps\le \frac{\vareps}{2}$ would suffice to guarantee $\vx^{k+1}$ to be an $\vareps$-KKT of \eqref{eq:ccp}. By the $\tilde\vareps$-dual feasibility condition and the assumption in \eqref{eq:def-D}, it holds for some $\tilde\nabla f_0(\vx^{k+1})\in \partial f_0(\vx^{k+1})$ that
\begin{align*}
D \tilde\vareps \ge & ~ \left\langle \textstyle \vx^{k+1}-\vx^k, \tilde\nabla f_0(\vx^{k+1}) + \vA^\top \vy^{k+1} + \sum_{i=1}^m z_i^{k+1}\nabla f_i(\vx^{k+1})+2\rho(\vx^{k+1}-\vx^k) \right\rangle \cr
\ge & ~\textstyle f_0(\vx^{k+1}) + \frac{3\rho}{2}\|\vx^{k+1}-\vx^k\|^2 - f_0(\vx^k) +(\vy^{k+1})^\top \vA(\vx^{k+1}-\vx^k) + \sum _{i=1}^m z_i^{k+1} \big(f_i(\vx^{k+1})-f_i(\vx^k)\big)\cr
\ge &~\textstyle f_0(\vx^{k+1}) + \frac{3\rho}{2}\|\vx^{k+1}-\vx^k\|^2 - f_0(\vx^k) -  2 \tilde\vareps \|\vp^{k+1}\|,
\end{align*} 
where the second inequality follows from the convexity of $f_0 + \frac{\rho}{2}\|\cdot - \vx^k\|^2$, and in the last inequality we have used the $\tilde\vareps$-feasibility of $\vx^k$ and $\vx^{k+1}$. Therefore, by bounding $\|\vp^{k+1}\|$, we have
\begin{equation}\label{eq:dec-obj}
\textstyle f_0(\vx^{k+1}) + \frac{3\rho}{2}\|\vx^{k+1}-\vx^k\|^2 - f_0(\vx^k) \le C \tilde\vareps ,
\end{equation}
where $C$ is a universal constant depending on $D$ and the bound of dual solutions. Summing the inequality over $k$, one can easily show that for any integer $K\ge1$,
$\frac{3\rho}{2}\sum_{k=0}^{K-1}\|\vx^{k+1}-\vx^k\|^2\le f_0(\vx^0)-f_0(\vx^K)+CK \tilde\vareps.$ 
Hence, we need to take $\tilde\vareps = O(\vareps^2)$ and $K=O(\vareps^{-2})$ to guarantee $\|\vx^{k+1}-\vx^k\|\le \frac{\vareps}{4\rho}$ for some $k\le K-1$. This way, we need $O\big(\vareps^{-3}|\log \vareps|\big)$ total number of gradient and function value evaluations, by using Theorem~\ref{thm:iter-eps-kkt}. This order of complexity result is comparable to that in \cite{kong2019complexity-pen}, which uses a quadratic penalty method, but it is worse than the result $O(\vareps^{-\frac{5}{2}}|\log \vareps|)$ in \cite{lin2019inexact-pp} whose method is also a quadratic penalty method. 

This fact motivates us to mix the iALM and PenMM, in order to achieve a better complexity result of $O(\vareps^{-\frac 5 2} |\log \vareps|)$. The method in \cite{lin2019inexact-pp} is a pure-penalty method. As we will show, our hybrid method can perform significantly better than a pure-penalty one. 

\subsection{Hybrid of iALM and PenMM}
In this subsection, we give the details of the hybrid method under the proximal point (PP) method framework. The mixture is not trivial. To have good numerical performance and a better complexity result, we must update the estimate of multipliers periodically but not too frequently by the iALM.
  
Suppose at the $k$-th iteration, the estimate of the primal solution is $\vx^k$. We form the $k$-th subproblem as
\begin{equation}\label{ksub}
\Min_{\vx\in\RR^n} \big\{f_0(\vx)+\rho \Vert \vx-\vx^k \Vert^2, \st \vA\vx=\vb, f_i(\vx)\le 0, i=1,\ldots, m\big\},
\end{equation}
and approximately solve it by either the iALM or PenMM method. The pseudocode is given in Algorithm~\ref{alg:ialm-noncvx}, which can be viewed as a multi-stage method. Due to the higher efficiency of the iALM over a penalty method, we approximately solve all the first $K_0$ subproblems by the iALM in the initial stage. This way, we can obtain a good estimate of the multipliers. Then we switch to the use of the PenMM method to have a better complexity result. To maintain the efficiency, we approximately solve a subproblem by the iALM at the end of each stage to update the estimate of the multipliers. It is important to notice that due to the use of an adaptive APG, our method only needs to know an upper bound on the weak-convexity constant $\rho$. {Practically, one can even use an under-estimate of $\rho$ to have fast convergence of the algorithm.} All other constants in Assumption~\ref{assump:composite},  including smoothness constants of $\{f_i\}$ and bound of $\dom(h)$, \color{black} are only required to exist, but their values are not needed.  Although there are three loops and several hyper-parameters in Algorithm~\ref{alg:ialm-noncvx}, it is not difficult to implement the algorithm efficiently in practice. 
 As a practical note, we find the following parameter choices universally good across all the problem instances that we test in section~\ref{sec:experiment}: $\beta_0 = 0.01, \sigma = 3, \gamma_1 = 2, \gamma_2 = 1.25, L_{\min} = \rho$.  \color{black} 

\begin{algorithm}[h] 
	{\small
		\caption{\textbf{H}ybrid \textbf{iA}LM and \textbf{Pe}nalty \textbf{M}ethod (HiAPeM) for weakly-convex cases of \eqref{eq:ccp}}\label{alg:ialm-noncvx}
		\DontPrintSemicolon
		\textbf{Input:} error tolerance $\vareps>0$, $\beta_0>0$, $\sigma>1$, and weak-convexity constant $\rho>0$;\;
		\textbf{Initialization:} choose $\vx^0\in\dom(h), \vy^0 = \vzero, \vz^0 = \vzero$, integer $K_0=N_0$ and $N_1$, $\gamma>1$, $0<\widehat\vareps_2\le\widehat\vareps_1 \le \vareps$, and $\gamma_1>1$, $\gamma_2\ge 1$,  $L_{\min}>0$;\;
		\For{$k=0,1,\ldots, K_0-1$}{
			\If{$k<K_0-1$}{
				Call Alg.~\ref{alg:ialm-stcvx}:  $(\beta_k,\vx^{k+1},\vp^{k+1})=\mathrm{iALM}\left(\frac{\widehat\vareps_1}{2},\beta_0,\sigma,\rho,\vx^k, L_{\min}, \gamma_1,\gamma_2\right)$; {\color{blue}\algorithmiccomment{obtain $\frac{\widehat\vareps_1}{2}$-KKT}}\; 
			}
			\If{$k=K_0-1$}{
				Call Alg.~\ref{alg:ialm-stcvx}: $(\beta_k,\vx^{k+1},\vp^{k+1})=\mathrm{iALM}\left(\frac{\widehat\vareps_1}{2},\beta_0,\sigma,\rho,\vx^k, L_{\min}, \gamma_1,\gamma_2\right)$;
			}
			\If{$\|\vx^{k+1}-\vx^k\|\le \frac{\vareps}{4\rho}$}{
				output $(\vx^{k+1}, \vp^{k+1})$ and stop
			}
		}
		Set $\bar\beta_{1,0}=\beta_k$ and $\bar\vp^1=\vp^{k+1}$; {\color{blue}\algorithmiccomment{store the estimated penalty parameter and the multipliers}}\; 
		Let $s=1$, $t=0$, $k=K_0$, and $K_s=K_0+ N_1$; {\color{blue}\algorithmiccomment{to start the $s$-th stage, $t$ counts inner iterations}}\;
		\While{$(\vx^k,\vp^k)$ is not $\vareps$-KKT to \eqref{eq:ccp}}{
			Call Alg.~\ref{alg:penalty}: $(\beta_k, \vx^{k+1},\vp^{k+1})=\mathrm{PenMM}(\frac{\widehat\vareps_2}{2},\bar\beta_{s,t},\sigma,\rho,\vx^k,\bar\vp^s,L_{\min}, \gamma_1,\gamma_2)$;{\color{blue}\algorithmiccomment{obtain $\frac{\widehat\vareps_2}{2}$-KKT of \eqref{ksub}}}\;
			\If{$\|\vx^{k+1}-\vx^k\|\le \frac{\vareps}{4\rho}$}{
				output $(\vx^{k+1},\vp^{k+1})$ and stop
			}
			Increase $t\gets t+1$, set $\bar\beta_{s,t}=\beta_k$, and increase $k\gets k+1$;\;
			\If{$k=K_s-1$}{
				Call Alg.~\ref{alg:ialm-stcvx}:  $(\beta_k,\vx^{k+1},\vp^{k+1})=\mathrm{iALM}\left( \frac{\widehat\vareps_1}{2},\beta_0,\sigma,\rho,\vx^k, L_{\min}, \gamma_1,\gamma_2\right)$;\; 
				\If{$\|\vx^{k+1}-\vx^k\|\le \frac{\vareps}{4\rho}$}{
					output $(\vx^{k+1},\vp^{k+1})$ and stop
				}
				Let $t=0$, $k=K_{s}$, $N_{s+1} = \lceil\gamma^{s} N_1\rceil$, and $K_{s+1}=K_{s}+N_{s+1}$;\;
				Increase $s \gets s + 1 $, set $\bar\beta_{s,0}=\beta_k$ and $\bar\vp^s=\vp^{k+1}$.{\color{blue}\algorithmiccomment{update penalty and multipliers from iALM}} 
			}
		}
	}
\end{algorithm}

\subsection{Complexity analysis of the hybrid method}

The complexity results for the iALM and PenMM both depend on the norm of the dual solution and the bound on the Lipschitz constant of the AL function. We first provide a uniform bound of dual solutions of all PP subproblems and on each estimated multiplier vector $\bar\vp^s$. 
The lemma below is directly from \cite[Lemma~3]{lin2019inexact-pp}.

\begin{lemma}[uniform bound on dual solutions]\label{lem:u-bound-dual}
	Under Assumptions~\ref{assump:composite} through \ref{assump:uniform-bound}, we have the following bound on the dual solution $(\vy_k^*, \vz_k^*)$ of \eqref{ksub}:  
	\begin{align*}
	\textstyle \Vert \vy_k^* \Vert &\leq M_{y} := \frac{\sqrt{\lambda_{\max}(\vA\vA^\top )}}{\lambda_{\min}(\vA\vA^\top )} (B_0+2\rho D+M)\left(1+\frac{D}{\dist(\vx_{\mathrm{feas}},\, \partial \dom(h))}+\frac{D\max_{i\ge1} B_i}{\min_{i\ge1} | f_i(\vx_{\mathrm{feas}}) |}\right) \\
	\Vert \vz_k^* \Vert &\leq M_{z} := \frac{D(B_0+\rho D+M)}{\min_{i\ge1} | f_i(\vx_{\mathrm{feas}}) |},
	\end{align*}
	where $\{B_i\}_{i=0}^m$ are the bounds of $|f_i|$'s, $\partial \dom(h)$ denotes the boundary of $\dom(h)$, and $\lambda_{\min}(\vA\vA^\top )$ and $\lambda_{\max}(\vA\vA^\top )$ respectively take the minimum and maximum eigenvalues of $\vA\vA^\top$.
\end{lemma}

Recalling $\vp=(\vy,\vz)$, we have from the above lemma that 
\begin{equation}\label{eq:bd-opt-pk}
\Vert \vp_k^* \Vert \leq M_p := \sqrt{M_y^2+M_z^2}.
\end{equation} 
With the above bound, we are able to bound each estimated multiplier vector $\bar\vp^s$ from the HiAPeM algorithm.

\begin{lemma}\label{lem:bd-est-p}
	Let $\{\bar\vp^s\}$ be from Algorithm~\ref{alg:ialm-noncvx}. Then $\frac{\|\bar\vp^s\|}{\bar\beta_{s,0}}\le \frac{\widehat\vareps_1}{2}$ and $\|\bar\vp^s\|\le M_{\widehat\vareps_1}$ for each $s\ge0$, where
	\begin{equation}\label{eq:def-M-eps}
	M_{\widehat\vareps_1} =2M_p+\sqrt{2\widehat\vareps_1\sigma \cdot \max\{\textstyle \frac{5}{4} M_p^2,\, M_p,\, \frac{\widehat\vareps_1}{4}\} }.
	\end{equation}
\end{lemma}

\begin{proof} 
We immediately have $\frac{\|\bar\vp^s\|}{\bar\beta_{s,0}}\le \frac{\widehat\vareps_1}{2}$ since $(\bar\beta_{s,0}, \bar\vp^s)$ is output by iALM with tolerance $\frac{\widehat\vareps_1}{2}$. Hence, we only need to show $\|\bar\vp^s\|\le M_{\widehat\vareps_1}$.

From \eqref{p-bound-kkt}, it follows that the output dual solution $\vp^{k+1}$ by iALM on the $k$-th subproblem \eqref{ksub} satisfies 
$\left\Vert \vp^{k+1} \right\Vert \leq 2 \left\Vert \vp_k^* \right\Vert + \sqrt{2\sum_{t=0}^{K-1}\beta_t \frac{\vareps_t^2}{\rho}}$, where $K$ is the total iteration number within the subroutine iALM. Notice that since we call iALM to return an $\frac{\widehat\vareps_1} 2$-KKT solution of \eqref{ksub}, applying the results in Theorem~\ref{thm:iter-eps-kkt}, we have 
$$\beta_t=\beta_0\sigma^t,\, \vareps_t=\sqrt\frac{\sigma-1}{\sigma+1}\frac{\widehat\vareps_1}4\min\{1,\sqrt{\rho}\},\, \text{ and }
K= \left\lceil\log_\sigma 
\max\left\{\frac{20\|\vp_k^*\|^2}{\beta_0\widehat\vareps_1},\, \frac{16\|\vp_k^*\|}{\beta_0\widehat\vareps_1},\, \frac{4}{\beta_0}\right\}\right\rceil + 1.$$ Hence, 
$$\sum_{t=0}^{K-1}\beta_t \frac{\vareps_t^2}{\rho}\le \sum_{t=0}^{K-1}\beta_t \frac{(\sigma-1)(\widehat\vareps_1)^2}{16(\sigma+1)}=\frac{(\sigma-1)(\widehat\vareps_1)^2\beta_0}{16(\sigma+1)}\frac{\sigma^K-1}{\sigma-1}\le \frac{\beta_0(\widehat\vareps_1)^2\sigma^{K-1}}{16}\le \widehat\vareps_1\sigma\max\big\{ \textstyle
\frac{5}{4}\|\vp_k^*\|^2,\, \|\vp_k^*\|, \frac{\widehat\vareps_1}{4}\big\}.$$
Since $\|\vp_k^*\|\le M_p$ from \eqref{eq:bd-opt-pk}, it holds that $\Vert \vp^{k+1} \Vert \le M_{\widehat\vareps_1}$ if the $k$-th subproblem \eqref{ksub} is solved by iALM, i.e., $\Vert \bar\vp^{s} \Vert \le M_{\widehat\vareps_1}$ for each $s\ge0$. This completes the proof.
\end{proof}

The next lemma gives the progress from the iALM steps.

\begin{lemma}[progress from iALM step]\label{lem:iter-by-ialm}
	Let $(\beta_k,\vx^{k+1},\vp^{k+1})$ be generated from Algorithm~\ref{alg:ialm-noncvx}. Then for each $0\le k < K_0-1$ and each $k=K_s-1$ for any $s\ge0$ (i.e., the output is by $\mathrm{iALM}$), it holds
	\begin{equation}\label{eq:obj-dec-from-ialm}
	f_0(\vx^{k+1})-f_0(\vx^k)+\frac{3\rho}{2}\|\vx^{k+1}-\vx^k\|^2 \le \big(\textstyle\frac{D}{2} + M_{\widehat\vareps_1}\big) \widehat\vareps_1,
	\end{equation}
	where $M_{\widehat\vareps_1}$ is defined in \eqref{eq:def-M-eps}, and $D$ is the diameter of $\dom(h)$.
\end{lemma}

\begin{proof} 
Notice that the output $(\beta_k,\vx^{k+1},\vp^{k+1})$ by $\mathrm{iALM}$ satisfies the $\frac{\widehat\vareps_1}{2}$-KKT conditions of \eqref{ksub}. Hence, it holds that
\begin{subequations}
\begin{align}
\textstyle\dist\left(\vzero, \partial_\vx\cL_0(\vx^{k+1},\vp^{k+1}) +2\rho(\vx^{k+1}-\vx^k)\right) \le \frac{\widehat\vareps_1}{2}, \label{eq:all-eps-pfeas}\\[0.1cm]
\textstyle\sqrt{\|\vA\vx^{k+1}-\vb\|^2+\|[\vf(\vx^{k+1})]_+\|^2} \le \frac{\widehat\vareps_1}{2}. \label{eq:all-eps-dfeas}
\end{align}
\end{subequations}
We have from \eqref{eq:all-eps-pfeas} and also the boundedness of $\dom(h)$ that for some subgradient $\tilde\nabla f_0(\vx^{k+1})$,
$$\left\langle\textstyle\vx^{k+1}-\vx^k, \tilde\nabla f_0(\vx^{k+1})+2\rho(\vx^{k+1}-\vx^k)+\vA^\top\vy^{k+1}+\sum_{i=1}^mz_i^{k+1}\nabla f_i(\vx^{k+1})\right\rangle \le \frac{\widehat\vareps_1 D}{2},$$
which together with the convexity of $f_0+\frac{\rho}{2}\|\cdot-\vx^k\|^2$ and $\{f_i\}_{i\ge 1}$ implies
\begin{align}\label{eq:ineq-f0-ialm-pt}
&~\textstyle f_0(\vx^{k+1})-f_0(\vx^k)+\frac{3\rho}{2}\|\vx^{k+1}-\vx^k\|^2 + \langle \vA\vx^{k+1}-\vb,\vy^{k+1}\rangle +\sum_{i=1}^m z_i^{k+1}f_i(\vx^{k+1})\nonumber\\[0.1cm]
\le &~ \textstyle \langle \vA\vx^{k}-\vb,\vy^{k+1}\rangle +\sum_{i=1}^m z_i^{k+1}f_i(\vx^{k}) + \frac{\widehat\vareps_1 D}{2}.
\end{align}
Since $\widehat\vareps_2\le\widehat\vareps_1$, $\vx^k$ must be $\frac{\widehat\vareps_1}{2}$-primal feasible, whichever subroutine (iALM or PenMM) is used to generate it. Hence, by Lemma~\ref{lem:bd-est-p} and the Young's inequality, we have
$$\textstyle \langle \vA\vx^{k}-\vb,\vy^{k+1}\rangle +\sum_{i=1}^m z_i^{k+1}f_i(\vx^{k}) - \left(\langle \vA\vx^{k+1}-\vb,\vy^{k+1}\rangle +\sum_{i=1}^m z_i^{k+1}f_i(\vx^{k+1})\right) \le \widehat\vareps_1 M_{\widehat\vareps_1}.$$
Therefore, adding the above inequality into \eqref{eq:ineq-f0-ialm-pt} gives the desired result.
\end{proof}


The next two lemmas give the progress from the PenMM steps.

\begin{lemma}\label{lem:sum-aug-terms}
	For each $s\ge 1$, it holds that 
	\begin{align}\label{eq:sum-aug-term}
	&~\textstyle \sum_{k=K_{s-1}}^{K_s-2}\left(\frac{\beta_{k}}{2}\|\vA\vx^{k}-\vb\|^2+\Psi_{\beta_{k}}(\vx^{k},\bar\vz^s)\right)-\sum_{k=K_{s-1}}^{K_s-2}\left(\frac{\beta_{k}}{2}\|\vA\vx^{k+1}-\vb\|^2+\Psi_{\beta_{k}}(\vx^{k+1},\bar\vz^s)\right)\nonumber\\
	\le &~ \frac{\widehat\vareps_1 M_{\widehat\vareps_1}}4 + \frac{\widehat\vareps_1}{2} + \frac{(\widehat\vareps_1)^2\sigma \beta(\widehat\vareps_1,\widehat\vareps_2)}{4}, 
	\end{align}
	where $M_{\widehat\vareps_1}$ is defined in \eqref{eq:def-M-eps}, $\Psi_\beta(\vx,\vz)=\frac{1}{2\beta}\left(\|[\vz+\beta\vf(\vx)]_+\|^2-\|\vz\|^2\right)$, and
	\begin{equation}\label{eq:def-beta-eps}
	\textstyle \beta(\widehat\vareps_1,\widehat\vareps_2) = \max\left\{\frac{24M_p^2 + 16 M_{\widehat\vareps_1}^2}{\widehat\vareps_2},\, \frac{8(M_p+M_{\widehat\vareps_1})}{\widehat\vareps_2},\, 8\right\}.
	\end{equation}
\end{lemma}

\begin{proof} 
Corresponding to the notation $\bar\beta_{s,t}$, we denote $\vx^{s,t}=\vx^k$ for each $k = t+K_{s-1}$. Then $(\bar\beta_{s,t+1},\vx^{s,t+1})$ is given by PenMM for each $0\le t \le N_s-2$, and the inequality in \eqref{eq:sum-aug-term} can be rewritten as 
\begin{align}\label{eq:sum-aug-term-equiv}
&~\sum_{t=0}^{N_s-2}\left(\frac{\bar\beta_{s,t+1}}{2}\|\vA\vx^{s,t}-\vb\|^2+\Psi_{\bar\beta_{s,t+1}}(\vx^{s,t},\bar\vz^s)\right)-\sum_{t=0}^{N_s-2}\left(\frac{\bar\beta_{s,t+1}}{2}\|\vA\vx^{s,t+1}-\vb\|^2+\Psi_{\bar\beta_{s,t+1}}(\vx^{s,t+1},\bar\vz^s)\right)\nonumber\\
\le &~ \frac{\widehat\vareps_1 M_{\widehat\vareps_1}}4 + \frac{\widehat\vareps_1}{2} + \frac{(\widehat\vareps_1)^2\sigma \beta(\widehat\vareps_1,\widehat\vareps_2)}{4}. 
\end{align}
We apply \eqref{eq:cond-betak-pen} with $\vareps=\frac{\widehat\vareps_2}{2}$, $\bar\vp^*=\bar\vp_k^*$, $\bar\vp=\bar\vp^s$, and also use \eqref{eq:bd-opt-pk} and Lemma~\ref{lem:bd-est-p} to have 
$$\beta_0\le \bar\beta_{s,t} < \sigma \beta(\widehat\vareps_1,\widehat\vareps_2),\text{ for each }1\le t\le N_s-1.$$

Noticing $\|\vA\vx^k-\vb\|\le \frac{\widehat\vareps_1}{2},\forall\, k$, and $\bar\beta_{s,t+1}\ge\bar\beta_{s,t}, \forall\, t\le N_s-1$, 
we have:
\begin{align}\label{eq:sum-axb}
&~\sum_{t=0}^{N_s-2}\left(\frac{\bar\beta_{s,t+1}}{2}\|\vA\vx^{s,t}-\vb\|^2-\frac{\bar\beta_{s,t+1}}{2}\|\vA\vx^{s,t+1}-\vb\|^2\right)\cr
=&~\frac{\bar\beta_{s,1}}{2}\|\vA\vx^{s,0}-\vb\|^2 - \frac{\bar\beta_{s,N_s-1}}{2}\|\vA\vx^{s,N_s-1}-\vb\|^2+\sum_{t=1}^{N_s-2}\left(\frac{\bar\beta_{s,t+1}}{2}\|\vA\vx^{s,t}-\vb\|^2-\frac{\bar\beta_{s,t}}{2}\|\vA\vx^{s,t}-\vb\|^2\right)\cr
\le&~\frac{\bar\beta_{s,1}}{2}\frac{(\widehat\vareps_1)^2}{4} +\sum_{t=1}^{N_s-2}\frac{\bar\beta_{s,t+1}-\bar\beta_{s,t}}{2}\frac{(\widehat\vareps_1)^2}{4} < \frac{(\widehat\vareps_1)^2}{8}\sigma\beta(\widehat\vareps_1,\widehat\vareps_2).
\end{align}

In addition, since $\Psi_\beta(\vx,\vz)=\frac{\beta}{2}\big\|[\vf(\vx)+\frac\vz\beta]_+\big\|^2-\frac{\|\vz\|^2}{2\beta}$, it holds 
{\small 
\begin{align}\label{eq:sum-psi-term}
&~\sum_{t=0}^{N_s-2}\left(\Psi_{\bar\beta_{s,t+1}}(\vx^{s,t},\bar\vz^s)-\Psi_{\bar\beta_{s,t+1}}(\vx^{s,t+1},\bar\vz^s)\right)\cr
=&~\sum_{t=0}^{N_s-2}\left(\frac{\bar\beta_{s,t+1}}{2}\left\|\Big[\vf(\vx^{s,t})+\frac{\bar\vz^s}{\bar\beta_{s,t+1}}\Big]_+\right\|^2-\frac{\bar\beta_{s,t+1}}{2}\left\|\Big[\vf(\vx^{s,t+1})+\frac{\bar\vz^s}{\bar\beta_{s,t+1}}\Big]_+\right\|^2\right)\cr
\le&~\frac{\bar\beta_{s,1}}{2}\left\|\Big[\vf(\vx^{s,0})+\frac{\bar\vz^s}{\bar\beta_{s,1}}\Big]_+\right\|^2
+\sum_{t=1}^{N_s-2}\left(\textstyle \frac{\bar\beta_{s,t+1}}{2}\left\|\textstyle \Big[\vf(\vx^{s,t})+\frac{\bar\vz^s}{\bar\beta_{s,t+1}}\Big]_+\right\|^2-\frac{\bar\beta_{s,t}}{2}\left\|\Big[\vf(\vx^{s,t})+\frac{\bar\vz^s}{\bar\beta_{s,t}}\Big]_+\right\|^2\right). 
\end{align}
}
For a fixed $t$, 
denote $I_+$  and $I_-$ as the index sets:
$$I_+=\big\{i\in[m]: f_i(\vx^{s,t})\ge0, \big\},\quad I_-=\big\{i\in[m]: -\bar z_i^s/\beta_{s,t+1}\le f_i(\vx^{s,t}) < 0\big\}.$$ 
Then
\begin{align}\label{eq:sum-psi-term2}
&~\frac{\bar\beta_{s,t+1}}{2}\left\|\Big[\vf(\vx^{s,t})+\frac{\bar\vz^s}{\bar\beta_{s,t+1}}\Big]_+\right\|^2-\frac{\bar\beta_{s,t}}{2}\left\|\Big[\vf(\vx^{s,t})+\frac{\bar\vz^s}{\bar\beta_{s,t}}\Big]_+\right\|^2\cr
\le&~\sum_{i\in I_+\cup I_-}\left(\frac{\bar\beta_{s,t+1}-\bar\beta_{s,t}}{2}[f_i(\vx^{s,t})]^2+\Big(\frac{1}{2\bar\beta_{s,t+1}}-\frac{1}{2\bar\beta_{s,t}}\Big)(\bar z_i^s)^2\right)\cr
\le&~\frac{\bar\beta_{s,t+1}-\bar\beta_{s,t}}{2} \sum_{i\in I_+}[f_i(\vx^{s,t})]^2 + \sum_{i\in I_-}\left(\frac{\bar\beta_{s,t+1}-\bar\beta_{s,t}}{2}\frac{(\bar z_i^s)^2}{(\bar\beta_{s,t+1})^2}+\Big(\frac{1}{2\bar\beta_{s,t+1}}-\frac{1}{2\bar\beta_{s,t}}\Big)(\bar z_i^s)^2\right)\cr
\le&~\frac{\bar\beta_{s,t+1}-\bar\beta_{s,t}}{2}\frac{(\widehat\vareps_1)^2}{4},
\end{align}
where 
in the last inequality, we have used $\|[\vf(\vx^k)]_+\|\le \frac{\widehat\vareps_1}{2},\forall\, k$, and the fact $\frac{\bar\beta_{s,t+1}-\bar\beta_{s,t}}{(\bar\beta_{s,t+1})^2}+\frac{1}{\bar\beta_{s,t+1}}-\frac{1}{\bar\beta_{s,t}}\le 0$. 
Furthermore,
\begin{align*}
\frac{\bar\beta_{s,1}}{2}\left\|\Big[\vf(\vx^{s,0})+\frac{\bar\vz^s}{\bar\beta_{s,1}}\Big]_+\right\|^2\le &~ \frac{\bar\beta_{s,1}}{2}\left\|\big[\vf(\vx^{s,0})\big]_++\frac{\bar\vz^s}{\bar\beta_{s,1}}\right\|^2 \\
=&~\frac{\bar\beta_{s,1}}{2}\big\|[\vf(\vx^{s,0})]_+\big\|^2 + (\bar\vz^s)^\top \big[\vf(\vx^{s,0})\big]_+ +  \frac{\|\bar\vz^s\|^2}{2\bar\beta_{s,1}} \\
\le &~ \frac{(\widehat\vareps_1)^2\bar\beta_{s,1}}8 + \frac{\widehat\vareps_1}{2} + \frac{\widehat\vareps_1 M_{\widehat\vareps_1}}4,
\end{align*}
where the the last inequality holds because $(\vx^{s,0}, \bar\vp^s)$ is the output of iALM with tolerance $\frac{\widehat\vareps_1}{2}$. 
Plugging the above inequality and also \eqref{eq:sum-psi-term2} into \eqref{eq:sum-psi-term} gives
\begin{align*}
\sum_{t=0}^{N_s-2}\left(\Psi_{\bar\beta_{s,t+1}}(\vx^{s,t},\bar\vz^s)-\Psi_{\bar\beta_{s,t+1}}(\vx^{s,t+1},\bar\vz^s)\right)
\le &~ 
\frac{\widehat\vareps_1 M_{\widehat\vareps_1}}4 + \frac{\widehat\vareps_1}{2} + \frac{(\widehat\vareps_1)^2\sigma \beta(\widehat\vareps_1,\widehat\vareps_2)}{8}.
\end{align*}
Now adding the above inequality to \eqref{eq:sum-axb}, 
we obtain the desired result. 
\end{proof}

\begin{lemma}[progress from PenMM step]\label{lem:iter-by-pen}
	Let $(\beta_k,\vx^{k+1},\vp^{k+1})$ be generated from Algorithm~\ref{alg:ialm-noncvx}. Then for each $s\ge1$, it holds
	\begin{equation}\label{eq:obj-sum-dec-from-pen}
	\begin{aligned}
	&~\textstyle f_0(\vx^{K_s-1}) - f_0(\vx^{K_{s-1}}) + \sum_{k=K_{s-1}}^{K_s-2} \rho\|\vx^{k+1}-\vx^k\|^2 
	\le \frac{5\widehat\vareps_1 M_{\widehat\vareps_1}}4 + \frac{\widehat\vareps_1}{2} + \frac{(\widehat\vareps_1)^2\sigma \beta(\widehat\vareps_1,\widehat\vareps_2)}{4} + \frac{(\widehat\vareps_2)^2}{4}(N_s-1).
	\end{aligned}
	\end{equation}
\end{lemma}

\begin{proof} 
For each $s\ge1$, we have from \eqref{eq:obj-dec-penalty} with $\vareps=\frac{\widehat\vareps_2}{2}$, $\bar\vx=\vx^k$, and $\bar\vp=\bar\vp^s$ that
\begin{align*}
&~f_0(\vx^{k+1})+\langle \vA\vx^{k+1}-\vb, \bar\vy^s\rangle + \frac{\beta_k}{2}\|\vA\vx^{k+1}-\vb\|^2+\Psi_{\beta_k}(\vx^{k+1},\bar\vz^s) + \rho\|\vx^{k+1}-\vx^k\|^2\\[0.1cm]
\le &~ f_0(\vx^k)+\langle \vA\vx^{k}-\vb, \bar\vy^s\rangle + \frac{\beta_k}{2}\|\vA\vx^{k}-\vb\|^2+\Psi_{\beta_k}(\vx^{k},\bar\vz^s) + \frac{(\widehat\vareps_2)^2}{4}.
\end{align*}
Sum up the above inequality over $k=K_{s-1}$ through $K_s-2$ and use \eqref{eq:sum-aug-term}. we have
\begin{align*}
&~f_0(\vx^{K_s-1}) - f_0(\vx^{K_{s-1}}) + \sum_{k=K_{s-1}}^{K_s-2} \rho\|\vx^{k+1}-\vx^k\|^2 \\
\le &~\langle \vA\vx^{K_{s-1}}-\vb, \bar\vy^s\rangle-\langle \vA\vx^{K_s-1}-\vb, \bar\vy^s\rangle + \frac{\widehat\vareps_1 M_{\widehat\vareps_1}}4 + \frac{\widehat\vareps_1}{2} + \frac{(\widehat\vareps_1)^2\sigma \beta(\widehat\vareps_1,\widehat\vareps_2)}{4} + \frac{(\widehat\vareps_2)^2}{4}(N_s-1).
\end{align*}
By the facts $\|\vA\vx^k-\vb\|\le \frac{\widehat\vareps_1}{2}$ and $\|\bar\vy^s\|\le M_{\widehat\vareps_1}$ for all $k$ and $s$, we obtain the desired result from the above inequality.
\end{proof}

By the progresses from the iALM and PenMM steps, we are able to bound the stage number as follows.

\begin{theorem}[bound on the stage number]\label{thm:bd-num-stage}
	Suppose that the conditions in Assumptions~\ref{assump:composite} through \ref{assump:uniform-bound} hold. Given $\vareps\in (0,\frac{1}{2}]$, let $\widehat\vareps_1 = \vareps$ and $\widehat\vareps_2 = \frac{\vareps}{2\sqrt 2}\min\big\{1, \frac {1} {\sqrt \rho}\big\}$ in Algorithm~\ref{alg:ialm-noncvx}. Assume $N_0\ge \frac{N_1}{\gamma-1}$. 
	Let $S$ be the smallest integer such that
	\begin{equation}\label{eq:cond-KS}
	\textstyle K_S\ge N_0+N_1\frac{\gamma^S-1}{\gamma-1}\ge \frac{96\rho B_{0}}{\vareps^2} + \frac{48\rho N_0 (\textstyle\frac{D}{2} + \widehat M\big)}{\vareps},\quad \text{and} \quad  \gamma^S \ge 2\frac{\widetilde C}{\vareps}\log\frac{\widetilde C}{\vareps},
	\end{equation}
	where $B_0$ the constant in Assumption~\ref{assump:uniform-bound}, $\widehat M = 2M_p+\sqrt{\sigma}\max\big\{\frac{\sqrt 5}{2}M_p,\, \sqrt{M_p},\, \textstyle\frac{1}{2\sqrt{2}}\big\}$, and
	\begin{equation}\label{eq:def-Cs}
	\widetilde C = \max\Big\{ \textstyle \frac{96\rho(\gamma-1)}{N_1\log\gamma}\left(\frac{1}2 + \frac{D}{2}+ \frac{9\widehat M}4 + \frac{ \widehat C\sigma}{4}
	\right),\, \vareps e\Big\},\quad \widehat C = \max\left\{\frac{16\sqrt{2}(3M_p^2+2\widehat M^2)}{\min\big\{1, \frac {1} {\sqrt \rho}\big\}},\, \frac{16\sqrt{2}(M_p+\widehat M)}{\min\big\{1, \frac {1} {\sqrt \rho}\big\}}, \, 4\right\},
	\end{equation}
	where $e\approx 2.718$ is the natural logarithmic base.
	Then Algorithm~\ref{alg:ialm-noncvx} must stop within $K_S$ iterations and the output $(\vx^{k+1},\vp^{k+1})$ is an $\vareps$-KKT point of \eqref{eq:ccp}. 
\end{theorem}

\begin{proof} 
Notice that whichever subroutine (iALM or PenMM) is called in Algorithm~\ref{alg:ialm-noncvx}, the output $(\vx^{k+1},\vp^{k+1})$ is an $\frac{\widehat\vareps_1}{2}$-KKT point of \eqref{ksub} since $\widehat\vareps_2\le \widehat\vareps_1$. In addition, because $\widehat\vareps_1\le \vareps$, $(\vx^{k+1},\vp^{k+1})$ will be an $\vareps$-KKT point of \eqref{eq:ccp} if $2\rho\|\vx^{k+1}-\vx^k\|\le \frac{\vareps}{2}$. In the following, we claim this must happen for some $k\le K_S-1$.

Sum up \eqref{eq:obj-dec-from-ialm} over $0\le k< K_0-1$ and $k=K_s-1$ for all $0\le s \le S$, and then add the resulting inequality to the summation of \eqref{eq:obj-sum-dec-from-pen} over $0\le s\le S$. We have
\begin{align*}
&f_0(\vx^{K_S})-f_0(\vx^0) + \rho\sum_{k=0}^{K_S-1}\|\vx^{k+1}-\vx^k\|^2 \le (K_0+S)\big(\textstyle\frac{D}{2} + M_{\widehat\vareps_1}\big) \widehat\vareps_1\\
&\hspace{2cm}+ S\left(\frac{5\widehat\vareps_1 M_{\widehat\vareps_1}}4 + \frac{\widehat\vareps_1}{2} + \frac{(\widehat\vareps_1)^2\sigma \beta(\widehat\vareps_1,\widehat\vareps_2)}{4}
\right) + \frac{(\widehat\vareps_2)^2}{4}(K_S-K_0-S).
\end{align*}
Plugging $\widehat\vareps_1 = \vareps$ and $\widehat\vareps_2 = \frac{\vareps}{2\sqrt 2}\min\big\{1, \frac {1} {\sqrt \rho}\big\}$ into the above inequality and noticing $M_{\widehat\vareps_1}\le \widehat M$ and $\vareps \beta(\widehat\vareps_1,\widehat\vareps_2)\le \widehat C$ with this choice of $\widehat\vareps_1$ and $\widehat\vareps_2$, we have
\begin{align*}
&f_0(\vx^{K_S})-f_0(\vx^0) + \rho\sum_{k=0}^{K_S-1}\|\vx^{k+1}-\vx^k\|^2 \le (K_0+S)\big(\textstyle\frac{D}{2} + \widehat M\big) \vareps\\
&\hspace{2cm}+ S\left(\frac{5\vareps \widehat M}4 + \frac{\vareps}{2} + \frac{\vareps\sigma \widehat C}{4}
\right) + \frac{\vareps^2}{32}(K_S-K_0-S)\textstyle\min\{1,\frac{1}{\rho}\}.
\end{align*}
Since $|f_0(\vx^{K_S})-f_0(\vx^0)|\le 2 B_0$ and  $\frac{K_S-K_0-S}{K_S}\le 1$, we have from the above inequality that
\begin{align}\label{eq:bd-sum-xk-term}
\rho^2\min_{0\le k < K_s}\|\vx^{k+1}-\vx^k\|^2\le \frac{2\rho B_{0}}{K_S}+\frac{\rho K_0\big(\textstyle\frac{D}{2} + \widehat M\big) \vareps}{K_S}+\frac{\rho\vareps S}{K_S}\left(\frac{1}2 + \frac{D}{2}+ \frac{9\widehat M}4 + \frac{ \widehat C\sigma}{4}\right)+\frac{\vareps^2}{32}.
\end{align}
From the condition on $K_S$ in \eqref{eq:cond-KS} and $K_0=N_0$, it follows that $\frac{2\rho B_{0}}{K_S}+\frac{\rho K_0(\frac{D}{2} + \widehat M) \vareps}{K_S}\le \frac{\vareps^2}{48}$. In addition, it holds $a\ge 2\log a$, or equivalently $\frac{\log(2a\log a)}{2a\log a} \le \frac{1}{a}$ for any $a\ge 1$. By this inequality with $a = \frac{\widetilde C}{\vareps}$, we have $\frac{\log \gamma^{\widetilde S}}{\gamma^{\widetilde S}} \le \frac{\vareps}{\widetilde C}$ for $\widetilde S>0$ such that $\gamma^{\widetilde S}=2a\log a$, and thus $\frac{\widetilde S}{\gamma^{\widetilde S}}\le \frac{\vareps}{\widetilde C\log\gamma}$. By the choice of $\widetilde C$, it holds $\frac{\widetilde C}{\vareps}\ge e$ and thus $\gamma^{\widetilde S}\ge 1$. Since $\frac{x}{\gamma^x}$ is decreasing with respect to $x\ge \frac{1}{\log\gamma}$, we have $\frac{S}{\gamma^S}\le \frac{\widetilde S}{\gamma^{\widetilde S}}\le \frac{\vareps}{\widetilde C\log\gamma}$ if $S$ satisfies the second condition in \eqref{eq:cond-KS}. Recall that $K_S\ge N_0+\frac{N_1(\gamma^S-1)}{\gamma-1}$. It holds $K_S\ge \frac{N_1\gamma^S}{\gamma-1}$ because $N_0\ge \frac{N_1}{\gamma-1}$. Hence, from the definition of $\widetilde C$ in \eqref{eq:def-Cs}, the inequality $\frac{S}{\gamma^S}\le \frac{\vareps}{\widetilde C\log\gamma}$ implies $\frac{\rho\vareps S}{K_S}\left(\frac{1}2 + \frac{D}{2}+ \frac{9\widehat M}4 + \frac{ \widehat C\sigma}{4}\right)\le \frac{\vareps^2}{96}$. Therefore, we have from \eqref{eq:bd-sum-xk-term} that $\min_{0\le k < K_S}\|\vx^{k+1}-\vx^k\|\le \frac{\vareps}{4\rho}$, i.e., Algorithm~\ref{alg:ialm-noncvx} must stop within $K_S$ iterations. 
\end{proof}

\begin{remark}
	The requirements $\vareps \le \frac{1}{2}$ and $N_0\ge \frac{N_1}{\gamma-1}$ in Theorem~\ref{thm:bd-num-stage} are for the convenience of analysis. We can have similar results for any $\vareps>0$ and any positive integers $N_0$ and $N_1$. In addition, we call the subroutines iALM and PenMM to obtain $\frac{\vareps}{2}$-KKT solutions of the subproblems. This is also for convenience of the analysis. Notice that the first three terms on the right hand side of \eqref{eq:bd-sum-xk-term} are constants or in the order of $\vareps$. Hence, for any constant $c\in(0,1)$, if the subroutines can return a $c\vareps$-KKT solution of each subproblem and the last term $\frac{\vareps^2}{32}$ in \eqref{eq:bd-sum-xk-term} is pushed to be less than $\frac{(1-c)^2\vareps^2}{4}$, we can still guarantee an $\vareps$-KKT solution of \eqref{eq:ccp} within the same order of complexity.
\end{remark}

Now we are ready to state the total complexity result.
\begin{theorem}[complexity results for nonconvex cases]\label{thm:iter-ialm-noncvx}
	Suppose that the conditions in Assumptions~\ref{assump:composite} through \ref{assump:uniform-bound} hold. Given $\vareps\in (0,\frac{1}{2}]$, let $\widehat\vareps_1 = \vareps$ and $\widehat\vareps_2 = \frac{\vareps}{2\sqrt 2}\min\big\{1, \frac {1} {\sqrt \rho}\big\}$ in Algorithm~\ref{alg:ialm-noncvx}. Assume $N_0\ge \frac{N_1}{\gamma-1}$. Then Algorithm~\ref{alg:ialm-noncvx} can produce an $\vareps$-KKT point of \eqref{eq:ccp} with at most $T_{\mathrm{total}}=(N_0+S)\lceil T_{\mathrm{iALM}} \rceil+ (K_S-N_0-S)\lceil T_{\mathrm{PenMM}}\rceil $ evaluations of gradient and function value, where $S$ is the smallest integer such that \eqref{eq:cond-KS} holds, and	
	\begin{align*}
	&T_{\mathrm{iALM}} = 6\sqrt{\gamma_1}\left( \sqrt{\kappa}K_{\mathrm{iALM}}+\widetilde{C}_{\vareps} \right)\log \left( \frac{2(1+\gamma_1)\widehat{L}_{\bar{\vareps}_1}D\sqrt{\frac{\gamma_1 \hat{L}_\vareps}{\rho}+1}}{\bar{\vareps}_1} \right) + \left( 3\widehat{E}+3 \right)K_{\mathrm{iALM}} , \\
	&T_{\mathrm{PenMM}} = 6\alpha_K\sqrt{\frac{\gamma_1}{\rho}}\log \left( \frac{2(1+\gamma_1)\bar{L}D\sqrt{\frac{\gamma_1 \bar{L}}{\rho}+1}}{\bar{\vareps}_2} \right) + \left(3\bar{E}+3\right)K_{\mathrm{PenMM}}.
	\end{align*}
	\color{black}
	Here, $\bar\vareps_1=\sqrt{\frac{\sigma-1}{\sigma+1}}\frac{\vareps}{2}\min\{1,\sqrt\rho\}$, $\bar\vareps_2=\frac{\vareps}{4\sqrt 2}\min\big\{\sqrt\rho, \frac {1} {\sqrt \rho}\big\}$, and
	\begin{align*}
	&\textstyle K_{\mathrm{iALM}}= \lceil\log_\sigma 
	\widehat C_\vareps\rceil + 1,\ K_{\mathrm{PenMM}}=\left\lceil \log_\sigma \frac{\bar\beta}{\beta_0}\right\rceil + 1,\ \kappa=\frac{L_0+2\rho+2M_p\sqrt{\textstyle\sum_{i=1}^mL_i^2}}{\rho}, \\
	&\textstyle \widehat C_\vareps=\max\left\{\frac{20 M_p^2}{\beta_0\vareps},\, \frac{16M_p}{\beta_0\vareps},\, \frac{4}{\beta_0}\right\},\quad \widehat L_\vareps=L_0+2\rho+\beta_0\sigma L_c \widehat C_\vareps +\sqrt{\sum_{i=1}^mL_i^2}\left(2 M_p +\frac{\vareps}{2\sqrt 2}\sqrt{\beta_0 \widehat C_\vareps}\right),\\
	&\textstyle \bar L= L_0+2\rho+L_c\sigma\bar\beta+\widehat M \sqrt{\sum_{i=1}^mL_i^2}, \quad \bar\beta =\max\left\{\frac{16\sqrt{2}(3M_p^2+2\widehat M^2)}{\vareps\min\big\{1, \frac {1} {\sqrt \rho}\big\}},\, \frac{16\sqrt{2}(M_p+\widehat M)}{\vareps\min\big\{1, \frac {1} {\sqrt \rho}\big\}}, \, 8\right\},\\
	& \textstyle \widetilde{C}_{\vareps} = \frac{\sqrt{\beta_0 L_c \sigma^2 \widehat C_\vareps}}{\sqrt{\rho}(\sqrt{\sigma}-1)}+ \left(\sqrt{\sum_{i=1}^{m}L_i^2} \frac{\vareps \sqrt{\beta_0}}{\sqrt{2\sigma}} \right)^{\frac{1}{2}} \frac{\sigma^{\frac{1}{2}}\widehat C_\vareps^{\frac{1}{4}}}{\sqrt{\rho}(\sigma^{\frac{1}{4}}-1)},\\
	& \textstyle \alpha_K = K_{\mathrm{PenMM}} \sqrt{2\rho}+K_{\mathrm{PenMM}} \left(L_0+ \widehat{M} \sqrt{\sum_{i=1}^{m}L_i^2}\right)^{\frac{1}{2}}+\frac{\sigma \sqrt{\bar{\beta}L_c}}{\sqrt{\sigma}-1},\\
	& \textstyle \widehat{E} = \left\lceil \log_{\gamma_1}\frac{\widehat{L}_{\vareps}}{\rho} \right\rceil,\quad \bar{E} = \left\lceil \log_{\gamma_1}\frac{\bar{L}}{\rho} \right\rceil,
	\end{align*}
	with $\widehat M$ defined in Theorem~\ref{thm:bd-num-stage}, $M_{p}$ in \eqref{eq:bd-opt-pk}, and $L_c$ in \eqref{eq:def-Lc-kappac}.
\end{theorem}

\begin{proof} 
From Theorem~\ref{thm:bd-num-stage}, it follows that Algorithm~\ref{alg:ialm-noncvx} must stop within $K_S$ iteration and produce an $\vareps$-KKT point of \eqref{eq:ccp}. In addition, notice that within the $K_S$ iterations, iALM is called $K_0+S$ times, and PenMM is called $K_S-K_0-S$ times. Hence, it suffices to bound the number of APG iterations that each call of iALM and PenMM takes. The formula of $T_{\mathrm{iALM}}$ directly follows from Theorem~\ref{thm:iter-eps-kkt}, with $\vareps$ replaced by $\frac\vareps 2$, $\mu$ replaced by $\rho$,  $L_0$ replaced by $L_0+2\rho$, and $\|\vp^*\|$ replaced by $M_p$ that is an upper bound for each $\|\vp_k^*\|$ in \eqref{eq:bd-opt-pk}. The formula of $T_{\mathrm{PenMM}}$ directly follows from Theorem~\ref{thm:complexity-pen}, with $\vareps$ replaced by $\frac{\vareps}{4\sqrt 2}\min\big\{1, \frac {1} {\sqrt \rho}\big\}$, $\|\bar\vp^*\|$ replaced by $M_p$, and $\|\bar\vp\|$ replaced by $\widehat M$ (since $\|\bar\vp^s\|\le \widehat M$ for each $s$).
\end{proof}

\begin{remark}
	Let $\vareps\in(0,1)$ be given. It is not difficult to see that $T_{\mathrm{iALM}}=O(\frac{1}{\sqrt{\vareps}}\log\frac{1}{\vareps})$ and $T_{\mathrm{PenMM}}=O(\frac{1}{\sqrt{\vareps}}\log\frac{1}{\vareps})$. Furthermore, it follows from \eqref{eq:cond-KS} that $K_S=O(\vareps^{-2})$. Hence, the overall complexity in terms of $\vareps$ is $O(\vareps^{-\frac{5}{2}}\log\frac{1}{\vareps})$. 	
	We can also investigate the dependence of $T_{\mathrm{total}}$ on other constants. Let $B_{\max} = \max_{ i \in [m]} (B_i+L_i)$, with $L_i$ and $B_i$ defined in \eqref{lip-i} and \eqref{P3-eqn-33}. Observe $T_{\mathrm{total}} = O(K_S)O(T_{\mathrm{PenMM}})$, $K_S = O(\frac{\rho B_0}{\vareps^2})$ and $T_{\mathrm{PenMM}} = O(\frac{\alpha_K}{\sqrt{\rho}}\log \vareps^{-1})$, and $\alpha_K = O(\sqrt{\bar{\beta}L_c}) = O(M_p (\Vert A \Vert + \sqrt{m} B_{\max}) \vareps^{-\frac{1}{2}})$. Hence, the overall complexity 
	is $O(M_p B_0 (\Vert A \Vert + \sqrt{m} B_{\max}) \sqrt{\rho} \vareps^{-\frac{5}{2}}\log\frac{1}{\vareps})$, where $B_0$ is defined in Assumption \ref{assump:uniform-bound}, and $M_p$ is given in \eqref{eq:bd-opt-pk}.
	\color{black}
\end{remark}

\section{Numerical experiments}\label{sec:experiment}

In this section, we do experiments to demonstrate the numerical performance of the proposed methods. The tests were conducted on the nonconvex linearly constrained quadratic program (LCQP), clustering by doubly stochastic matrix problem (CLUS), and nonconvex quadratically constrained quadratic program (QCQP). 
We compare the proposed methods 
to the QP-AIPP in \cite{kong2019complexity-pen} and the iPPP in \cite{lin2019inexact-pp}. All the tests were performed in MATLAB 2017a on an iMAC with 4 cores and 8GB memory.


\subsection{Experiments on nonconvex linearly constrained quadratic programs}\label{sec:qp}
In this subsection, we test\footnote{More tests and more numerical results can be found in the longer arXiv version \cite{li2020augmented}.} the proposed method on solving nonconvex linearly constrained quadratic program (LCQP):
\begin{equation}\label{eq:ncQP}
\min_{\vx \in \mathbb{R}^n}  \frac{1}{2} \vx^\top \vQ \vx + \vc^\top  \vx,\ 
\text{s.t. } \vA\vx=\vb,\ x_i \in [l_i,u_i],\,\forall\, i\in [n],
\end{equation}
where $\vA \in \mathbb{R}^{m\times n}$, and $\vQ\in\RR^{n\times n}$ is symmetric and indefinite (thus the objective is nonconvex). In the test, we generated all data randomly. The smallest eigenvalue of $\vQ$ is $-\rho < 0$, and thus the problem is $\rho$-weakly convex. For all tested instances, we set $l_i=0$ and $u_i=5$ for each $i\in [n]$.

We generated one group of LCQP instances of size $m=100$ and $n=1,000$. It consists of three subgroups corresponding to $\rho\in\{0.1, 1, 10\}$. In general, the bigger $\rho$ is, the harder the problem is. We tested the HiAPeM in Algorithm~\ref{alg:ialm-noncvx} and compared it to the QP-AIPP method in \cite{kong2019complexity-pen}. For HiAPeM, we set $N_1=2$ and tried different values of $N_0\in \{1, 10, 100\}$. Recall that $N_0$ is the number of calling the iALM as the subroutine in the initial stage. In general, the iALM is better than a pure-penalty method on solving the same functional constrained problem. We expected that the more the iALM was called in the initial stage, the sooner a solution of a desired-accuracy would be obtained. In addition, we set $\gamma=1.1$ and $\beta_k=\sigma^k\beta_0$ with $\sigma=3$ and $\beta_0=0.01$ for all instances. {Since we can compute the smoothness constant of each subproblem in Algorithm~\ref{alg:ialm-noncvx}, we set $L_{\min}$ to the computed smoothness constant, and the line search was never performed in Algorithm~\ref{alg2}.} For QP-AIPP, we set its parameters as required by the theorems in \cite{kong2019complexity-pen}. However, we must relax the stopping condition in \cite[Eq. (53)]{kong2019complexity-pen} by using a tolerance larger than required, because otherwise the number $A_k$ in its ACG subroutine may overflow in MATLAB (larger than $10^{308}$). For this same reason, we must set each $l_i=0$ in all instances. We noticed that if we set $l_i$ to a negative number, QP-AIPP also had the overflow issue. The targeted accuracy was set to $\vareps=10^{-3}$ for all instances.

For HiAPeM, we report the primal residual, dual residual, running time (in second), and the number of gradient evaluation, shortened as \verb|pres|, \verb|dres|, \verb|time|, and \verb|#Grad|, respectively. For QP-AIPP, we also report the number of objective evaluation \verb|#Obj|. Notice that HiAPeM does not need to evaluate the objective value, and thus its \verb|#Obj| equals 0. 
The results are shown in Tables~\ref{table:qp-large}. We make three observations from the results. First, all the compared methods took more time as $\rho$ increased, and this indicates that bigger $\rho$ yields harder problems. Second, the proposed HiAPeM performs better with larger $N_0$, which coincides with our expectation. Thirdly, for each subgroup of instances, the proposed method is faster than the QP-AIPP method on average.

\setlength{\tabcolsep}{3pt}

\begin{table}\caption{Results by the proposed algorithm HiAPeM with $N_1 = 2$ and three different choices of $N_0$ and by the QP-AIPP method in \cite{{kong2019complexity-pen}} on solving a $\rho$-weakly convex QP \eqref{eq:ncQP} of size $m=100$ and $n=1000$, where $\rho \in\{0.1,1.0,10\}$.}\label{table:qp-large}
	\begin{center}
		\resizebox{\textwidth}{!}{
			\begin{tabular}{|c||cccc|cccc|cccc|ccccc|}
				\hline
				trial & pres & dres & time & \#Grad & pres & dres & time & \#Grad & pres & dres & time & \#Grad & pres & dres & time & \#Obj & \#Grad\\\hline\hline 
				&\multicolumn{4}{|c|}{HiAPeM with $N_0=100$} & \multicolumn{4}{|c|}{HiAPeM with $N_0=10$} & \multicolumn{4}{|c|}{HiAPeM with $N_0=1$} & \multicolumn{5}{|c|}{QP-AIPP}\\\hline
				\multicolumn{18}{|c|}{weak convexity constant: $\rho=0.1$}\\\hline
				1 	 & 6.05e-4 & 8.02e-4 & 270.06 & 779651 & 5.55e-4 & 6.20e-4 & 456.31 & 1839499 & 6.37e-4 & 7.50e-4 & 529.08 & 2252436 & 7.86e-4 & 3.02e-5 & 396.45 & 773458 & 773469\\
				2 	 & 8.36e-5 & 9.33e-4 & 114.10 & 314697 & 3.55e-4 & 9.74e-4 & 172.44 & 668977 & 7.78e-4 & 5.66e-4 & 555.81 & 2401904 & 8.72e-4 & 2.26e-5 & 399.29 & 787592 & 787603\\
				3 	 & 5.25e-5 & 9.42e-4 & 166.54 & 493007 & 4.77e-4 & 6.95e-4 & 265.95 & 1157951 & 5.31e-4 & 5.30e-4 & 270.64 & 1223062 & 8.18e-4 & 6.51e-5 & 322.37 & 640008 & 640019\\
				4 	 & 1.83e-4 & 5.32e-4 & 234.24 & 686201 & 8.04e-4 & 9.44e-4 & 312.89 & 1292083 & 3.14e-4 & 6.84e-4 & 379.33 & 1640810 & 9.98e-4 & 2.92e-5 & 224.27 & 449510 & 449521\\
				5 	 & 5.70e-5 & 9.06e-4 & 111.45 & 292782 & 3.72e-4 & 3.95e-4 & 177.32 & 666678 & 6.44e-5 & 9.36e-4 & 231.67 & 939466 & 6.42e-4 & 7.68e-5 & 357.85 & 716670 & 716681\\\hline
				avg. & 1.96e-4 & 8.23e-4 & 179.28 & 513268 & 5.13e-4 & 7.26e-4 & 276.98 & 1125038 & 4.65e-4 & 6.93e-4 & 393.31 & 1691536 & 8.23e-4 & 4.48e-5 & 340.05 & 673448 & 673459\\\hline
				\multicolumn{18}{|c|}{weak convexity constant: $\rho=1.0$}\\\hline
				1 	 & 7.99e-6 & 9.80e-4 & 445.52 & 1485421 & 1.44e-4 & 7.00e-4 & 530.74 & 2285016 & 2.46e-5 & 9.62e-4 & 556.64 & 2439467 & 7.33e-4 & 2.30e-5 & 1087.35 & 2199916 & 2199929\\
				2 	 & 6.26e-5 & 9.99e-4 & 175.13 & 381934 & 5.01e-5 & 9.42e-4 & 220.31 & 884453 & 3.92e-4 & 9.28e-4 & 325.25 & 1389783 & 8.43e-4 & 2.50e-5 & 1483.14 & 2993824 & 2993837\\
				3 	 & 2.85e-5 & 9.45e-4 & 246.58 & 628786 & 8.52e-6 & 8.06e-4 & 698.68 & 2937267 & 4.52e-4 & 8.89e-4 & 965.49 & 4131220 & 7.63e-4 & 2.53e-5 & 1252.89 & 2534358 & 2534371\\
				4 	 & 2.08e-5 & 8.96e-4 & 197.90 & 421591 & 8.60e-6 & 8.14e-4 & 363.88 & 1464673 & 1.20e-5 & 6.74e-4 & 621.74 & 2548256 & 8.34e-4 & 2.46e-5 & 1294.27 & 2617348 & 2617361\\
				5 	 & 3.23e-4 & 9.66e-4 & 515.85 & 1825923 & 4.66e-6 & 6.57e-4 & 762.97 & 3265620 & 3.73e-4 & 7.10e-4 & 1520.34 & 6690393 & 7.26e-4 & 2.28e-5 & 722.80 & 1457100 & 1457113\\\hline
				avg. & 8.86e-5 & 9.58e-4 & 316.20 & 948731 & 4.31e-5 & 7.84e-4 & 515.32 & 2167406 & 2.51e-4 & 8.32e-4 & 797.89 & 3439824 & 7.80e-4 & 2.42e-5 & 1168.09 & 2360509 & 2360522\\\hline
				\multicolumn{18}{|c|}{weak convexity constant: $\rho=10$}\\\hline
				1 	 & 4.63e-6 & 8.85e-4 & 765.01 & 2773432 & 6.91e-5 & 9.37e-4 & 1109.39 & 4715496 & 3.00e-4 & 9.41e-4 & 1176.39 & 5072715 & 9.24e-4 & 8.10e-5 & 1067.83 & 2163914 & 2163930\\
				2 	 & 3.31e-4 & 8.57e-4 & 1023.72 & 4152972 & 3.71e-5 & 7.86e-4 & 1277.90 & 5806754 & 4.21e-4 & 9.35e-4 & 1134.51 & 5112349 & 9.89e-4 & 8.84e-5 & 1519.88 & 3087706 & 3087722\\
				3 	 & 4.04e-5 & 8.48e-4 & 420.37 & 1347740 & 3.08e-4 & 9.90e-4 & 717.14 & 3112172 & 3.95e-5 & 7.84e-4 & 776.35 & 3402113 & 9.44e-4 & 8.67e-5 & 1293.43 & 2625136 & 2625152\\
				4 	 & 1.12e-5 & 9.45e-4 & 255.00 & 561620 & 4.52e-5 & 9.32e-4 & 461.22 & 1890121 & 1.06e-5 & 9.55e-4 & 536.69 & 2270553 & 9.67e-4 & 8.96e-5 & 1846.89 & 3746176 & 3746192\\
				5 	 & 2.05e-5 & 9.32e-4 & 412.18 & 1245989 & 5.05e-5 & 9.90e-4 & 1052.02 & 4472420 & 3.60e-5 & 9.86e-4 & 1102.85 & 4703986 & 8.98e-4 & 8.83e-5 & 2014.29 & 4087306 & 4087322\\\hline
				avg. & 8.15e-5 & 8.94e-4 & 575.26 & 2016351 & 1.02e-4 & 9.27e-4 & 923.53 & 3999393 & 1.61e-4 & 9.20e-4 & 945.36 & 4112343 & 9.44e-4 & 8.68e-5 & 1548.46 & 3142048 & 3142064\\\hline
			\end{tabular}
		}
	\end{center}
\end{table}

\subsection{Experiments on clustering by doubly stochastic matrix problem}\label{sec:clus}

In this subsection, we test the proposed method HiAPeM and compare it to the QP-AIPP method in \cite{kong2019complexity-pen} on the clustering by doubly stochastic matrix problem (CLUS) given in \cite{yang2012clustering}: 
\begin{equation}\label{eq:clus_sp}
\min_{\vX = \vX^\top, \vX \ge 0} -\sum_{i,j}A_{ij}\log \sum_k x_{ik} x_{jk} - (\alpha - 1)\sum_{i,j}\log x_{ij},
\text{ s.t. } \vX \vone = \vone,
\end{equation}
where $\alpha > 1$, 
and $\vA \in \RR^{n \times n}$ is a symmetric doubly stochastic matrix, i.e., $\vA\in\mathbb{SP}^n := \{\vX \in \RR^{n\times n}\, |\, \vX \ge 0, \vX = \vX^\top, \vX \vone = \vone\}$. 
Observe that any feasible $\vX \in \RR^{n \times n}$ of \eqref{eq:clus_sp} is a doubly stochastic matrix. In the test, we chose $\alpha\in \{1.2, 2, 5\}$ and $n\in\{100, 500\}$, and we generated the matrix $\vA$ approximately uniformly at random from $\mathbb{SP}^n$. If all the components of $\vX$ are greater than a positive number, the first term in the objective of \eqref{eq:clus_sp} is weakly convex. However, different from \eqref{eq:ncQP}, it is not trivial to compute the weak convexity constant $\rho$ of \eqref{eq:clus_sp}. Hence, we tuned it by picking the best one from $\{1, 10, 100\}$ such that that the tested algorithm could converge fastest. This way, we set $\rho=1$ for HiAPeM and $\rho=100$ for QP-AIPP. In addition, we set $L_{\min}=\rho$, $\gamma_1=2$, and $\gamma_2=1.25$ for HiAPeM. 

We used the target accuracy $\vareps=10^{-3}$ for both methods and stopped them once an $\vareps$-KKT point was reached. For HiAPeM, we adopted two settings, with $N_0 = 10$ and $N_0 = 1$ respectively, and set other parameters to the same values as those in the previous subsection. We also tested HiAPeM with $N_0 = 100$, but we noticed that for the CLUS problem, it could always stop within several stages, and thus this setting would produce the same result as that with $N_0 = 10$. For all trials, we report the primal residual, dual residual, complementarity violation, running time (in second), number of objective evaluation, and number of gradient evaluation, which are evaluated at the last iterate and shortened as \verb|pres|, \verb|dres|, \verb|compl|, \verb|time|, \verb|#Obj|, and \verb|#Grad|, respectively. 
The results for all trials are shown in Table~\ref{table:clus}, where for the instance with $(\alpha, n)=(5, 500)$, QP-AIPP did not stop within 2 hour. From the results, 
we see that to produce an $\vareps$-KKT point, the proposed HiAPeM requires significantly fewer evaluations of objective and gradient than QP-AIPP in all trials except the one with $(\alpha, n)=(1.2, 100)$. In addition, we notice that an instance with a bigger $\alpha$ appears easier for HiAPeM but harder for QP-AIPP. This is probably because different values of $\rho$ were used for the two methods.

\begin{table}\caption{Results by the proposed algorithm HiAPeM with $N_1 = 2$ and two different choices of $N_0$ and by the QP-AIPP method in \cite{{kong2019complexity-pen}} on solving instances of the clustering problem \eqref{eq:clus_sp} with different pairs of  $(\alpha, n)$. For the instance with $(\alpha, n)=(5, 500)$, QP-AIPP could not stop within 2 hour.}\label{table:clus}
	\begin{center}
		\resizebox{\textwidth}{!}{
			\begin{tabular}{|c||ccccc|ccccc|ccccc|ccccc|}
				\hline
			instance	 & pres & dres & time & \#Obj & \#Grad & pres & dres & time & \#Obj & \#Grad & pres & dres & time & \#Obj & \#Grad \\\hline\hline 
$(\alpha, n)$				&\multicolumn{5}{|c|}{HiAPeM with $N_0=10$} & \multicolumn{5}{|c|}{HiAPeM with $N_0=1$} & \multicolumn{5}{|c|}{QP-AIPP} \\\hline
$(1.20, 100)$ & 1.88e-5 & 4.04e-4 & 10.63 & 23436 & 18179 & 1.88e-5 & 3.60e-4 & 23.90 & 54928 & 44598 & 5.37e-4 & 6.65e-8 & 11.51 & 27842 & 13935\\
$(2.00, 100)$ & 1.21e-4 & 1.97e-4 & 8.91 & 23436 & 18176 & 1.21e-4 & 1.32e-4 & 9.80 & 33160 & 25847 & 6.23e-4 & 1.73e-8 & 20.95 & 49944 & 24988\\
$(5.00, 100)$ & 3.58e-4 & 2.21e-4 & 7.14 & 19286 & 14919 & 3.58e-4 & 1.10e-4 & 6.99 & 25598 & 19863 & 6.13e-4 & 2.36e-7 & 33.76 & 84326 & 42181\\\hline
$(1.20, 500)$ & 2.70e-4 & 2.49e-4 & 122.22 & 12808 & 9868 & 2.70e-4 & 3.15e-4 & 145.03 & 18284 & 14211 & 8.70e-4 & 4.90e-6 & 1061.0 & 103158 & 51598\\
$(2.00, 500)$ & 1.92e-5 & 4.90e-4 & 102.27 & 12100 & 9227 & 1.92e-5 & 4.47e-4 & 118.25 & 13418 & 10627 & 5.35e-4 & 2.05e-5 & 1674.2 & 167468 & 83756\\
$(5.00, 500)$ & 6.95e-5 & 4.10e-4 & 66.80 & 8376 & 6371 & 1.52e-4 & 4.57e-4 & 71.35 & 8506 & 6750 & -- & -- & -- & -- & --
\\\hline
			\end{tabular}
		}
	\end{center}
\end{table}


\subsection{Experiments on nonconvex quadratically constrained quadratic program}\label{sec:qcqp}

In this subsection, we test the proposed HiAPeM method in Algorithm~\ref{alg:ialm-noncvx} and compare it to the iPPP in \cite{lin2019inexact-pp} on the nonconvex quadratically constrained quadratic program (QCQP):
\begin{equation}\label{eq:qcqp}
\min_{\vx \in \mathbb{R}^n} \frac{1}{2} \vx^\top \vQ_{0}\vx + \vc_{0}^\top  \vx,
\text{ s.t. } \frac{1}{2} \vx^\top \vQ_{j}\vx + \vc_j^\top  \vx + d_j \leq 0,\ \forall\, j \in [m];\ x_i \in [l_i,u_i],\ \forall\, i\in [n].
\end{equation}
Here, $\vQ_j$ is positive semidefinite, $\forall j \in \{1, \cdots, m\}$, but $\vQ_0$ is indefinite with the smallest eigenvalue $ -\rho < 0$. Hence, the objective is $\rho$-weakly convex. 
In the test, we set $m = 10$ and $n = 1000$, and all matrices and vectors were generated randomly. For each $i$, we set $l_i=-5$ and $u_i=5$. For each $j$, $d_j$ was made negative, and thus the Slater's condition holds. 

Three groups of QCQP instances were generated, with $\rho = 0.1$, $\rho = 1$, and $\rho = 10$ respectively. Each group consists of 5 independent trials. 
We set $\vareps=10^{-3}$ for both methods and stopped them once an $\vareps$-KKT point was produced. For HiAPeM, we adopted three settings, with $N_0 = 100, N_0 = 10$ and $N_0 = 1$ respectively, where $N_0$ is the number of iALM calls in the initial stage of Algorithm~\ref{alg:ialm-noncvx}. In addition, we set $N_1 = 2$, $\beta_0=0.01$ and $\sigma=3$. 
For the iPPP method, we set $\beta_k=\beta_0 \sqrt{k}$, where $\beta_0$ was tuned to $10$ in order to have fast convergence. 
For all instances, we report the primal residual, dual residual, complementarity violation, running time (in second), number of evaluations on objective and gradient, which are evaluated at the last iterate and shortened as \verb|pres|, \verb|dres|, \verb|compl|, \verb|time|, \verb|#Obj|, and \verb|#Grad|, respectively. 
The results for all instances are shown in Table~\ref{table:qcqp}. 
From the results, we see that to produce a target-accurate KKT point, our proposed method (in all the adopted settings) requires significantly fewer objective and gradient evaluations than the iPPP method in all trials. This advantage is more significant as $\rho$ gets larger (that indicates harder problems) when $N_0=100$ or $N_0=10$ was set for the proposed HiAPeM method. 

\begin{table}\caption{Results by the proposed algorithm HiAPeM with $N_1 = 2$ and three different choices of $N_0$ and by the iPPP method in \cite{{lin2019inexact-pp}} on solving instances of $\rho$-weakly convex QCQP \eqref{eq:qcqp} of size $m=10$ and $n=1000$, where $\rho \in\{0.1,1.0,10\}$.}\label{table:qcqp}
	\begin{center}
		\resizebox{\textwidth}{!}{
			\begin{tabular}{|c||ccccc|ccccc|ccccc|ccccc|}
				\hline
				trial & pres & dres & time & \#Obj & \#Grad & pres & dres & time & \#Obj & \#Grad & pres & dres & time & \#Obj & \#Grad & pres & dres & time & \#Obj & \#Grad\\\hline\hline 
				&\multicolumn{5}{|c|}{HiAPeM with $N_0=100$} & \multicolumn{5}{|c|}{HiAPeM with $N_0=10$} & \multicolumn{5}{|c|}{HiAPeM with $N_0=1$} & \multicolumn{5}{|c|}{iPPP in \cite{lin2019inexact-pp}}\\\hline
				\multicolumn{21}{|c|}{weak convexity constant: $\rho=0.1$}\\\hline
				1 	 & 7.22e-5 & 4.81e-4 & 38.75 & 8524 & 7360 & 7.22e-5 & 4.81e-4 & 39.15 & 8524 & 7360 & 8.14e-5 & 2.29e-4 & 91.12 & 19712 & 17149 & 1.00e-3 & 8.31e-4 & 691.86 & 155636 & 131196\\
				2 	 & 8.80e-5 & 5.44e-4 & 38.30 & 8360 & 7217 & 8.80e-5 & 5.44e-4 & 37.81 & 8360 & 7217 & 7.94e-5 & 2.63e-4 & 87.79 & 19268 & 16758 & 9.99e-4 & 8.52e-4 & 676.70 & 152798 & 128625\\
				3 	 & 6.20e-5 & 5.27e-4 & 37.90 & 8352 & 7210 & 6.20e-5 & 5.27e-4 & 37.80 & 8352 & 7210 & 1.06e-4 & 2.40e-4 & 83.60 & 18654 & 16220 & 9.99e-4 & 8.67e-4 & 683.97 & 153926 & 129947\\
				4 	 & 8.68e-5 & 5.26e-4 & 40.61 & 8452 & 7301 & 8.69e-5 & 5.26e-4 & 38.35 & 8452 & 7301 & 1.02e-4 & 2.27e-4 & 87.47 & 19436 & 16907 & 9.99e-4 & 8.81e-4 & 682.32 & 153350 & 129038\\
				5 	 & 7.58e-5 & 5.05e-4 & 39.08 & 8652 & 7471 & 7.58e-5 & 5.05e-4 & 39.22 & 8652 & 7471 & 8.51e-5 & 2.30e-4 & 87.16 & 19494 & 16956 & 1.00e-3 & 7.83e-4 & 704.86 & 159150 & 134166\\\hline
				avg. & 7.70e-5 & 5.17e-4 & 38.93 & 8468 & 7312 & 7.70e-5 & 5.17e-4 & 38.47 & 8468 & 7312 & 9.08e-5 & 2.38e-4 & 87.42 & 19313 & 16798 & 9.99e-4 & 8.43e-4 & 687.94 & 154972 & 130594\\\hline
				\multicolumn{21}{|c|}{weak convexity constant: $\rho=1.0$}\\\hline
				1 	 & 2.84e-4 & 6.88e-4 & 63.03 & 13900 & 11882 & 2.84e-4 & 9.73e-4 & 53.14 & 11784 & 10091 & 1.42e-6 & 9.28e-4 & 182.94 & 41008 & 35613 & 9.99e-4 & 7.91e-4 & 1071.7 & 244976 & 204995\\
				2 	 & 3.82e-4 & 9.38e-4 & 60.71 & 13564 & 11605 & 1.41e-5 & 8.51e-4 & 52.73 & 11436 & 9795 & 1.04e-5 & 7.94e-4 & 178.61 & 40246 & 34980 & 1.00e-3 & 7.69e-4 & 1034.2 & 235678 & 195468\\
				3 	 & 3.10e-4 & 8.41e-4 & 64.71 & 14352 & 12291 & 3.10e-4 & 8.05e-4 & 55.51 & 12266 & 10512 & 1.34e-5 & 7.15e-4 & 182.82 & 40840 & 35471 & 9.99e-4 & 8.79e-4 & 1079.6 & 246566 & 205662\\
				4 	 & 3.20e-4 & 7.88e-4 & 65.21 & 14506 & 12420 & 3.20e-4 & 7.55e-4 & 55.95 & 12414 & 10639 & 1.10e-5 & 6.96e-4 & 185.46 & 41090 & 35691 & 1.00e-3 & 8.63e-4 & 1087.5 & 248176 & 206763\\
				5 	 & 3.04e-4 & 8.23e-4 & 62.75 & 14356 & 12287 & 3.04e-4 & 7.80e-4 & 54.53 & 12308 & 10545 & 1.48e-5 & 7.08e-4 & 178.87 & 41266 & 35844 & 1.00e-3 & 8.74e-4 & 1078.7 & 246562 & 205520\\\hline
				avg. & 3.20e-4 & 8.16e-4 & 63.28 & 14136 & 12097 & 2.46e-4 & 8.33e-4 & 54.37 & 12042 & 10316 & 1.02e-5 & 7.68e-4 & 181.74 & 40890 & 35520 & 1.00e-3 & 8.35e-4 & 1070.3 & 244392 & 203682\\\hline
				\multicolumn{21}{|c|}{weak convexity constant: $\rho=10$}\\\hline
				1 	 & 3.15e-4 & 8.47e-4 & 117.85 & 26796 & 22728 & 4.07e-6 & 9.20e-4 & 147.38 & 33670 & 28964 & 4.48e-6 & 9.16e-4 & 972.67 & 201724 & 176373 & 1.75e-3 & 8.01e-4 & 5232.3 & 1195082 & 912498\\
				2 	 & 4.17e-4 & 8.70e-4 & 113.41 & 25892 & 21993 & 4.31e-6 & 9.79e-4 & 142.64 & 32752 & 28176 & 5.02e-6 & 7.78e-4 & 857.67 & 198256 & 173336 & 1.82e-3 & 2.48e-4 & 5024.1 & 1135550 & 861027\\
				3 	 & 3.05e-4 & 9.54e-4 & 116.00 & 26406 & 22401 & 4.37e-6 & 8.80e-4 & 160.75 & 37004 & 31887 & 2.73e-6 & 8.96e-4 & 866.56 & 200770 & 175511 & 1.77e-3 & 2.38e-4 & 4935.1 & 1131958 & 857562\\
				4 	 & 4.26e-4 & 8.47e-4 & 117.89 & 26964 & 22906 & 5.84e-6 & 8.10e-4 & 147.36 & 33912 & 29186 & 2.47e-6 & 7.97e-4 & 880.45 & 203596 & 177998 & 1.79e-3 & 2.86e-4 & 4805.8 & 1102934 & 832237\\
				5 	 & 3.59e-4 & 9.18e-4 & 114.93 & 26192 & 22216 & 3.38e-6 & 8.79e-4 & 146.80 & 33816 & 29095 & 2.89e-6 & 8.75e-4 & 882.05 & 204200 & 178541 & 1.77e-3 & 4.51e-4 & 5307.9 & 1218004 & 932373\\\hline
				avg. & 3.64e-4 & 8.87e-4 & 116.02 & 26450 & 22449 & 4.39e-6 & 8.94e-4 & 148.99 & 34231 & 29462 & 3.52e-6 & 8.52e-4 & 891.88 & 201709 & 176352 & 1.78e-3 & 4.05e-4 & 5061.0 & 1156706 & 879139\\\hline
			\end{tabular}
		}
	\end{center}
\end{table}

\subsection{Performance sensitivity on algorithm parameters}
In Algorithm~\ref{alg:ialm-noncvx}, there are several hyper-parameters. Theoretically, we can arbitrarily choose them to satisfy the required conditions. However, a few efforts are often needed to tune them to have the ``best'' performance. In this subsection, we demonstrate, by more experimental results, how the choices of the parameters affect the performance of the proposed method. $L_{\min}, \gamma_1$ and $\gamma_2$ are the input parameters in the core subroutine Algorithm~\ref{alg2}. 
We use one QCQP instance to show how these parameters affect the overall performance of HiAPeM. 

We generated a QCQP instance of size $n = 1000, m = 10$ and with a weak convexity constant $\rho = 1$. 
We simply set $N_0 = 100, N_1 = 2$ and the target accuracy $\vareps=10^{-3}$, and we varied $L_{\min} \in \{ 1, 5, 10 \}, \gamma_1 \in \{ 1.5, 2, 2.5 \}, \gamma_2 = \{1.1, 1.25, 2\}$. This leads to $27$ different combinations. 
The results are shown in Table~\ref{table:param}. For each $L_{\min}$, we highlight the best pair of $(\gamma_1,\gamma_2)$ in \textbf{bold}, and for each $\gamma_2$, we highligh the best pair of $(L_{\min},\gamma_1)$ in \textbf{\emph{italic bold}}. From the results, we see that the performance of the algorithm changes with respect to the choices of parameters, but it is not severely sensitive, at least the numbers of objective/gradient evaluations are all in the same magnitude. 

\begin{table}\caption{Results by HiAPeM with $N_0 = 100$ and $N_1 = 2$ on solving a $1$-weakly convex QCQP \eqref{eq:clus_sp} of size $m = 10$ and $n=1000$, with $L_{\min} \in \{ 1, 5, 10 \}, \gamma_1 \in \{ 1.5, 2, 2.5 \}, \gamma_2 = \{1.1, 1.25, 2\}$}\label{table:param}
	\begin{center}
		\resizebox{\textwidth}{!}{
			\begin{tabular}{|c|ccccc||c|ccccc||c|ccccc|}
				\hline
				$(L_{\min},\gamma_1, \gamma_2)$ & pres & dres & time & \#Obj & \#Grad & 
				$(L_{\min},\gamma_1, \gamma_2)$ & pres & dres & time & \#Obj & \#Grad &
				$(L_{\min},\gamma_1, \gamma_2)$ & pres & dres & time & \#Obj & \#Grad
				\\\hline\hline 
				$(1,1.5,1.1)$ 	 & 2.79e-4 & 7.39e-4 & 60.31 & 13372 & 11449 &
				$(1,1.5,1.25)$ 	 & 2.78e-4 & 7.30e-4 & 70.07 & 16270 & 12849 &
				$(1,1.5,2)$ 	 & 2.76e-4 & 7.47e-4 & 116.06 & 27090 & 18287 \\
				$(1,2,1.1)$ 	 & 2.79e-4 & 7.46e-4 & 49.65 & 11496 & 10424 &
				$(1,2,1.25)$ 	 & 2.79e-4 & 7.34e-4 & 57.47 & 13286 & 11351 &
				$(1,2,2)$ 	 & 2.64e-4 & 7.60e-4 & 92.23 & 21738 & 16114 \\
				$\textbf{(1,2.5,1.1)}$ 	 & 2.72e-4 & 7.42e-4 & \textbf{48.76} & \textbf{11278} & \textbf{10469} &
				$(1,2.5,1.25)$ 	 & 2.80e-4 & 7.37e-4 & 55.32 & 12432 & 10974 &
				$(1,2.5,2)$ 	 & 2.81e-4 & 7.40e-4 & 70.76 & 16610 & 12878 \\\hline
				$(5,1.5,1.1)$ 	 & 2.76e-4 & 7.43e-4 & 55.68 & 12896 & 11206 &
				$(5,1.5,1.25)$ 	 & 2.75e-4 & 7.32e-4 & 68.00 & 15992 & 12784 &
				$(5,1.5,2)$ 	 & 2.76e-4 & 7.25e-4 & 114.28 & 26706 & 18099 \\
				$(5,2,1.1)$ 	 & 2.78e-4 & 7.25e-4 & 49.06 & 11426 & 10455 &
				$(5,2,1.25)$ 	 & 2.75e-4 & 7.33e-4 & 55.89 & 13034 & 11214 &
				$(5,2,2)$ 	 & 2.78e-4 & 7.30e-4 & 81.24 & 19028 & 14145 \\
				$\textbf{(5,2.5,1.1)}$ 	 & 2.79e-4 & 7.42e-4 & \textbf{47.59} & \textbf{11066} & \textbf{10347} &
				$(5,2.5,1.25)$ 	 & 2.81e-4 & 7.38e-4 & 54.01 & 12202 & 10844 &
				$\textbf{(\emph{5,2.5,2})}$ 	 & 2.75e-4 & 7.32e-4 & \textbf{\emph{69.69}} & \textbf{\emph{16456}} & \textbf{\emph{12813}} \\\hline
				$(10,1.5,1.1)$ 	 & 2.79e-4 & 7.36e-4 & 54.43 & 12678 & 11087 &
				$(10,1.5,1.25)$ 	 & 2.80e-4 & 7.43e-4 & 66.80 & 15702 & 12611 &
				$(10,1.5,2)$ 	 & 2.78e-4 & 7.48e-4 & 113.37 & 26628 & 18082 \\
				$(10,2,1.1)$ 	 & 2.78e-4 & 7.25e-4 & 48.71 & 11318 & 10401 &
				$(10,2,1.25)$ 	 & 2.75e-4 & 7.33e-4 & 55.62 & 12926 & 11160 &
				$(10,2,2)$ 	 & 2.78e-4 & 7.30e-4 & 80.00 & 18920 & 14091 \\
				$\textbf{(\emph{10,2.5,1.1})}$ 	 &2.76e-4 &7.28e-4 &\textbf{\emph{47.16}} &\textbf{\emph{10948}} &\textbf{\emph{10282}} &
				$\textbf{(\emph{10,2.5,1.25})}$ 	 & 2.74e-4 & 7.38e-4 & \textbf{\emph{52.48}} & \textbf{\emph{12232}} & \textbf{\emph{10906}} &
				$(10,2.5,2)$ 	 & 2.74e-4 & 7.26e-4 & 71.46 & 16420 & 12811 \\\hline
			\end{tabular}
		}
	\end{center}
\end{table}

\section{Concluding remarks}\label{sec:conclusion} 

We have established an iteration complexity result in a nonergodic sense of an AL-based first-order method for solving nonlinear weakly-convex functional constrained problems. Given $\vareps>0$, our method requires $O(\vareps^{-\frac{5}{2}}|\log {\vareps}|)$ proximal gradient steps to produce an $\vareps$-KKT solution. The result is so far the best and in the same order of that achieved by a pure-penalty method. Numerically, we demonstrated the significant superiority of our method over a pure-penalty method.

\bibliographystyle{abbrv}
\bibliography{optim}

\end{document}

%% file: macros.tex
\usepackage{mathtools}

\usepackage[normalem]{ulem} 

\newcommand{\va}{{\mathbf{a}}}
\newcommand{\vb}{{\mathbf{b}}}
\newcommand{\vc}{{\mathbf{c}}}

\newcommand{\vf}{{\mathbf{f}}}

\newcommand{\vp}{{\mathbf{p}}}

\newcommand{\vt}{{\mathbf{t}}}

\newcommand{\vx}{{\mathbf{x}}}
\newcommand{\vy}{{\mathbf{y}}}
\newcommand{\vz}{{\mathbf{z}}}

\newcommand{\vA}{{\mathbf{A}}}

\newcommand{\vQ}{{\mathbf{Q}}}

\newcommand{\vX}{{\mathbf{X}}}

\newcommand{\cB}{{\mathcal{B}}}

\newcommand{\cL}{{\mathcal{L}}}

\newcommand{\cN}{{\mathcal{N}}}

\newcommand{\cP}{{\mathcal{P}}}

\newcommand{\cX}{{\mathcal{X}}}

\newcommand{\vareps}{\varepsilon}

\newcommand{\RR}{\mathbb{R}} 
\newcommand{\vzero}{\mathbf{0}} 
\newcommand{\vone}{{\mathbf{1}}} 

\newcommand{\dist}{\mathrm{dist}}    
\newcommand{\dom}{{\mathrm{dom}}} 

\newcommand{\st}{\mbox{ s.t. }}

\DeclareMathOperator*{\argmin}{arg\,min} 

\DeclareMathOperator*{\Min}{minimize}

\newcommand{\bc}{\begin{center}}
\newcommand{\ec}{\end{center}}

\newcommand{\bdm}{\begin{displaymath}}
\newcommand{\edm}{\end{displaymath}}

\newcommand{\beq}{\begin{equation}}
\newcommand{\eeq}{\end{equation}}

\newcommand{\bfl}{\begin{flushleft}}
\newcommand{\efl}{\end{flushleft}}

\newcommand{\bt}{\begin{tabbing}}
\newcommand{\et}{\end{tabbing}}

\newcommand{\beqn}{\begin{eqnarray}}
\newcommand{\eeqn}{\end{eqnarray}}

\newcommand{\beqs}{\begin{align*}} 
\newcommand{\eeqs}{\end{align*}}  

\newtheorem{assumption}{Assumption}


%% file: ConvgOfiALMforNLP.bbl
\begin{thebibliography}{10}

\bibitem{aybat2013augmented}
N.~S. Aybat and G.~Iyengar.
\newblock An augmented lagrangian method for conic convex programming.
\newblock {\em arXiv preprint arXiv:1302.6322}, 2013.

\bibitem{boob2019proximal}
D.~Boob, Q.~Deng, and G.~Lan.
\newblock Proximal point methods for optimization with nonconvex functional
  constraints.
\newblock {\em arXiv preprint arXiv:1908.02734}, 2019.

\bibitem{hamedani2018primal}
E.~Y. Hamedani and N.~S. Aybat.
\newblock A primal-dual algorithm for general convex-concave saddle point
  problems.
\newblock {\em arXiv preprint arXiv:1803.01401}, 2018.

\bibitem{hestenes1969multiplier}
M.~R. Hestenes.
\newblock Multiplier and gradient methods.
\newblock {\em Journal of optimization theory and applications}, 4(5):303--320,
  1969.

\bibitem{kong2019complexity-pen}
W.~Kong, J.~G. Melo, and R.~D. Monteiro.
\newblock Complexity of a quadratic penalty accelerated inexact proximal point
  method for solving linearly constrained nonconvex composite programs.
\newblock {\em SIAM Journal on Optimization}, 29(4):2566--2593, 2019.

\bibitem{lan2013iteration-pen}
G.~Lan and R.~D. Monteiro.
\newblock Iteration-complexity of first-order penalty methods for convex
  programming.
\newblock {\em Mathematical Programming}, 138(1-2):115--139, 2013.

\bibitem{lan2016iteration-alm}
G.~Lan and R.~D. Monteiro.
\newblock Iteration-complexity of first-order augmented lagrangian methods for
  convex programming.
\newblock {\em Mathematical Programming}, 155(1-2):511--547, 2016.

\bibitem{li2019-piALM}
F.~Li and Z.~Qu.
\newblock An inexact proximal augmented lagrangian framework with arbitrary
  linearly convergent inner solver for composite convex optimization.
\newblock {\em arXiv preprint arXiv:1909.09582}, 2019.

\bibitem{li2020augmented}
Z.~Li and Y.~Xu.
\newblock Augmented lagrangian based first-order methods for convex and
  nonconvex programs: nonergodic convergence and iteration complexity.
\newblock {\em arXiv preprint arXiv:2003.08880}, 2020.

\bibitem{lin2019inexact-pp}
Q.~Lin, R.~Ma, and Y.~Xu.
\newblock Inexact proximal-point penalty methods for non-convex optimization
  with non-convex constraints.
\newblock {\em arXiv preprint arXiv:1908.11518}, 2019.

\bibitem{lin2014adaptive}
Q.~Lin and L.~Xiao.
\newblock An adaptive accelerated proximal gradient method and its homotopy
  continuation for sparse optimization.
\newblock In {\em International Conference on Machine Learning}, pages 73--81,
  2014.

\bibitem{liu2019nonergodic}
Y.-F. Liu, X.~Liu, and S.~Ma.
\newblock On the nonergodic convergence rate of an inexact augmented lagrangian
  framework for composite convex programming.
\newblock {\em Mathematics of Operations Research}, 44(2):632--650, 2019.

\bibitem{lu2018iteration}
Z.~Lu and Z.~Zhou.
\newblock Iteration-complexity of first-order augmented lagrangian methods for
  convex conic programming.
\newblock {\em arXiv preprint arXiv:1803.09941}, 2018.

\bibitem{ma2019proximally}
R.~Ma, Q.~Lin, and T.~Yang.
\newblock Proximally constrained methods for weakly convex optimization with
  weakly convex constraints.
\newblock {\em arXiv preprint arXiv:1908.01871}, 2019.

\bibitem{necoara2013rate}
I.~Necoara and V.~Nedelcu.
\newblock Rate analysis of inexact dual first-order methods application to dual
  decomposition.
\newblock {\em IEEE Transactions on Automatic Control}, 59(5):1232--1243, 2013.

\bibitem{nedelcu2014computational}
V.~Nedelcu, I.~Necoara, and Q.~Tran-Dinh.
\newblock Computational complexity of inexact gradient augmented lagrangian
  methods: application to constrained mpc.
\newblock {\em SIAM Journal on Control and Optimization}, 52(5):3109--3134,
  2014.

\bibitem{powell1969method}
M.~J. Powell.
\newblock {\em A method for non-linear constraints in minimization problems}.
\newblock in Optimization, R. Fletcher Ed., Academic Press, New York, NY, 1969.

\bibitem{rigollet2011neyman}
P.~Rigollet and X.~Tong.
\newblock Neyman-pearson classification, convexity and stochastic constraints.
\newblock {\em Journal of Machine Learning Research}, 12(Oct):2831--2855, 2011.

\bibitem{rockafellar2015convex}
R.~T. Rockafellar.
\newblock {\em Convex analysis}.
\newblock Princeton university press, 1970.

\bibitem{rockafellar1973dual}
R.~T. Rockafellar.
\newblock A dual approach to solving nonlinear programming problems by
  unconstrained optimization.
\newblock {\em Mathematical programming}, 5(1):354--373, 1973.

\bibitem{sahin2019inexact}
M.~F. Sahin, A.~Alacaoglu, F.~Latorre, V.~Cevher, et~al.
\newblock An inexact augmented lagrangian framework for nonconvex optimization
  with nonlinear constraints.
\newblock In {\em Advances in Neural Information Processing Systems}, pages
  13943--13955, 2019.

\bibitem{scott2005neyman}
C.~Scott and R.~Nowak.
\newblock A neyman-pearson approach to statistical learning.
\newblock {\em IEEE Transactions on Information Theory}, 51(11):3806--3819,
  2005.

\bibitem{xu2017first}
Y.~Xu.
\newblock First-order methods for constrained convex programming based on
  linearized augmented lagrangian function.
\newblock {\em INFORMS Journal on Optimization (online first)}, 2021.

\bibitem{xu2019iter-ialm}
Y.~Xu.
\newblock Iteration complexity of inexact augmented lagrangian methods for
  constrained convex programming.
\newblock {\em Mathematical Programming, Series A}, (185):199--244, 2021.

\bibitem{yang2012clustering}
Z.~Yang and E.~Oja.
\newblock Clustering by low-rank doubly stochastic matrix decomposition.
\newblock In {\em Proceedings of the 29th International Coference on
  International Conference on Machine Learning}, pages 707--714, 2012.

\end{thebibliography}
